\documentclass[reqno]{amsart}
\usepackage[usenames, dvipsnames]{color}
\usepackage{amsthm}
\usepackage{amsmath}
\usepackage{amssymb}
\usepackage{amscd}
\usepackage{graphics}
\usepackage{latexsym}
\usepackage{stmaryrd}
\usepackage{hyperref}
\usepackage{empheq}

\theoremstyle{remark}

\definecolor{titlecol}{named}{BrickRed}
\definecolor{headcol}{named}{Violet}
\definecolor{seccol}{named}{Red}
\definecolor{sseccol}{named}{Bittersweet}
\definecolor{pbcol}{named}{Black}
\definecolor{sncol}{named}{Brown}
\definecolor{acol1}{named}{Red}
\definecolor{acol2}{named}{Apricot}

\def\p{\partial}
\def\R{\mathbb{R}}

\def\l{\lambda}

\def\cC{{\mathcal C}}

\def\cF{{\mathcal F}}

\def\cL{{\mathcal L}}

\def\cR{{\mathcal R}}
\def\cS{{\mathcal S}}

\theoremstyle{plain}
\newtheorem{thm}{Theorem}
\newtheorem{prop}{Proposition}[section]
\newtheorem{lemma}[prop]{Lemma}
\newtheorem{cor}[prop]{Corollary}
\newtheorem{rmk}[prop]{Remark}

\newtheorem{defn}[prop]{Definition}
\newtheorem{conj}[prop]{Conjecture}

\numberwithin{prop}{section}

\begin{document}

\title{The Gursky-Streets equations}
\dedicatory{Dedicated to Professor Sir S. Donaldson on his 60th birthday}
\author{Weiyong He}
\address{Department of Mathematics, University of Oregon, Eugene, OR 97403. }
\email{whe@uoregon.edu}

\begin{abstract}Gursky-Streets \cite{GS2} introduced a formal Riemannian metric on the space of conformal metrics in a fixed conformal class of a compact Riemannian four-manifold 
in the context of the $\sigma_2$-Yamabe problem. The geodesic equation of Gursky-Streets' metric is a fully nonlinear degenerate elliptic equation and Gursky-Streets have proved uniform $C^{0, 1}$ regularity for a perturbed equation. Gursky-Streets apply the results and parabolic smoothing of Guan-Wang flow to show that the solution of $\sigma_2$-Yamabe problem is unique. 
A key ingredient is the convexity of Chang-Yang's $\cF$-functional along the (smooth) geodesic, in view of Gursky-Streets metric and a weighted Poincare inequality of B. Andrews on manifolds with positive Ricci curvature.
In this paper we establish uniform $C^{1, 1}$ regularity of the Gursky-Streets' equation. 
As an application, we can establish strictly the geometric structure in terms of Gursky-Streets' metric, in particular the convexity of $\cF$-functional along $C^{1, 1}$ geodesic.  This in particular gives a straightforward proof of the uniqueness of solutions of $\sigma_2$-Yamabe problem. 
\end{abstract}

\maketitle

\section{Introduction}

Recently Gursky-Streets \cite{GS2}  introduced a new formal Riemannian metric on the space of  conformal metrics in a fixed conformal class of a compact Riemannian four-manifold 
in the context of the $\sigma_2$-Yamabe problem. The Gursky-Streets metric has many remarkable properties and as an application, Gursky-Streets proved that  solutions of the $\sigma_2$-Yamabe problem are \emph{unique}, unless the manifold is conformally equivalent to the round four-sphere. A key ingredient is to solve a fully nonlinear degenerate elliptic equation, arising as the \emph{geodesic} equation of the Gursky-Streets metric. Their strategy is inspired by the theory of the space of K\"ahler metrics (in a fixed K\"ahler class). In 1980s Mabuchi \cite{M1, M2} introduced a formal Riemannian metric on the space of K\"ahler metrics in a fixed K\"ahler class, which is now called the \emph{Mabuchi metric}. Donaldson \cite{Donaldson97} set up a program  in 1990s to study the geometry of the space of K\"ahler metrics and its various applications to the well-known problems in K\"ahler geometry, notably the existence and uniqueness of Calabi's extremal K\"ahler metrics \cite{C1} (constant scalar curvature metrics). Donaldson's program and related problems have great impact to the K\"ahler geometry. A key ingredient is the geodesic equation, which can be written as a homogeneous complex Monge-Ampere equation by the work of Semmes \cite{Semmes} and Donaldson \cite{Donaldson97}. A foundational result  is to solve the geodesic equation (the Dirichlet problem) by X. Chen \cite{Chen}, proving the existence of $C^{1, \bar 1}$ solution for any two given boundary datum. Since then there are tremendous work on the study of the space of K\"ahler metrics and related problems,  see  \cite{BB, Blocki, CZ, Chen08, CLP, CS, CT, CVW, D, DR, He, PS} for example and reference therein for vastly growing papers in literature.

The geometry of Gursky-Streets' metric on the space of conformal metrics in a fixed conformal class of a compact Riemannian four-manifold 
in the context of the $\sigma_2$-Yamabe problem has a parallel theory as the geometry of the space of K\"ahler metrics. 
We briefly recall Gursky-Streets' set up and results and refer readers to their paper for detailed discussions. Let $(M, [g])$ be a compact Riemannian manifold of dimension $n$ ($n\geq 3$) with a fixed conformal class $[g]$. We write $Ric$ as the Ricci tensor of $g$. The Schouten tensor of a given metric is defined to be, 
\[
A:=\frac{1}{n-2}\left(Ric-\frac{1}{2(n-2)} Rg\right)
\]
The $\sigma_k$-curvature is defined to be the $k$-symmetric function of the eigenvalues of $g^{-1}A$. For $k=1$, 
\[
\sigma_1(g^{-1}A)=\frac{R}{2(n-1)}. 
\]
For $1\leq k\leq n$, $A\in\Gamma^+_k$ if $\sigma_j(g^{-1}A)>0$ for all $1\leq j\leq k$. The main interest in \cite{GS2} is when $n=4$ and $k=2$, in the context of $\sigma_2$-Yamabe problem. 
Let $(M^4, [g_0])$ is a compact four manifold with a fixed conformal class, such that $A_{g_0}\in \Gamma^+_2$. Such a metric necessarily has positive Ricci curvature, by a result of Guan-Wang-Viaclovsky \cite{GWV}. 
Denote
\[
\mathcal{C}^{+}=\mathcal{C}^{+}[g]=\{g_u=e^{-2u}g: A_u\in \Gamma^+_{2}\}.
\]
The $\sigma_2$-Yamabe problem is to seek a a metric $g_u=e^{-2u}g\in [g]$ such that 
\begin{equation}\label{stwo}
\sigma_2(g_u^{-1}A_u)\equiv \text{const}.
\end{equation}
For surveys on solving the $\sigma_k$-Yamabe problem for general
$2\leq k \leq n$ see \cite{Viaclovsky} and \cite{STW}. When $n=4$, the existence of solutions to \eqref{stwo} in a conformal class with nonempty $\cC^+$ was proved by Chang-Gursky-Yang \cite{CGY}.
It turns out that the $\sigma_k$ problem has a variational structure for $n=2k$ \cite{BV}. In particular, such a solution is a critical point of the $\cF$-functional defined in \cite{CY},
\begin{equation}\label{F-functional}
\begin{split}
\cF(u)=&\int \left(2\Delta u|\nabla u|^2-|\nabla u|^4-2Ric(\nabla u, \nabla u)+R|\nabla u|^2-8u\sigma_2(A_g)\right)dV\\
&-2 \int \sigma_2(A_g) dV\log{\left(\text{Vol}^{-1}\int e^{4u}dV\right)}
\end{split}
\end{equation}
Gursky-Streets \cite{GS2} defined and studied a metric on the space of $\mathcal{C}^{+}$ by,
\begin{equation}\label{gsmetric}
\langle \psi, \phi\rangle_{u}=\int_M \phi \psi \sigma_2(g_u^{-1}A(g_u)) dV_u
\end{equation}
We briefly summarize their main results. 

\begin{thm}[Gursky-Streets] Given $(M^2, g)$ with $A_g\in \Gamma^+_2$, \eqref{gsmetric} defines a metric of nonpositive sectional curvature on $\cC^+$. Given $u: [0, 1]\times M\rightarrow \R$ such that $g_u=e^{-2u}g$ with $A_u\in \Gamma^+_2$, the geodesic equation is of the form
\begin{equation}\label{geodesic}
u_{tt}-\frac{1}{\sigma_2(A_u)}\langle T_1(A_u), \nabla u_t\otimes \nabla u_t\rangle=0
\end{equation}
where $T_1(A_u)$ is the first Newton transformation of $A_u$ and $\langle \cdot, \cdot \rangle$ denotes the inner product of tensor bundles induced by the background metric  $g$.  
\end{thm}

The geometry of $\cC^+$ with Gursky-Streets metric gives a nice geometric insight of the variational structure of Chang-Yang's functional $\cF$ and this leads naturally to the uniqueness of the solutions of $\sigma_2$-Yamabe problem.
\begin{thm}[Gursky-Streets] Let $(M^4, g)$ be a compact four manifold with nonempty $\cC^+$.  Then $\cF$-functional is formally geodesically convex. Moreover
\begin{enumerate}\item There exists a unique solution to the $\sigma_2$-Yamabe problem in $[g]$ if $(M^4, g)$ is not conformally equivalent to the round $S^4$.

\item In $[g_{S^4}]$, all solutions to the $\sigma_2$-problem are round metrics.
\end{enumerate}
\end{thm}

The argument of the uniqueness theorem in \cite{GS2} is intricate and technically very involved. It consists two main steps. Gursky-Streets \cite{GS2} first proved the existence of a unique smooth solution the perturbed equation (given two boundary values) with uniform $C^1$ estimates, 
\[
(1+\epsilon) u_{tt}\sigma_2(A_u)-\langle T_1(A_u), \nabla u_t\otimes \nabla u_t\rangle =f
\]
with positive $f$ and $\epsilon$ (they actually studied the equation for general $n$ and $k\leq n$). The $C^2$ estimates depend on $\epsilon^{-1}$ in particular. The lack of $C^{1, 1}$ regularity causes lots of technical difficulty to argue the uniqueness. To overcome such a difficulty, Gursky-Streets ran a parabolic $\sigma_2$-flow  (the Guan-Wang flow \cite{GW}) for an approximate geodesic and proved uniform estimates along the Guan-Wang flow. With this parabolic smoothing and properties of $\cF$-functional along the Guan-Wang flow, Gursky-Streets were able to prove the uniqueness theorem.

Our main interest is to study  the degenerate equation \eqref{geodesic}, or more specifically the perturbed equation with a smooth function $f>0$, 
\begin{equation}\label{geodesic1}
u_{tt}\sigma_2(A_u)-\langle T_1(A_u), \nabla u_t\otimes \nabla u_t\rangle=f
\end{equation}
Our main result is to confirm the desired expectation that \eqref{geodesic1} admits a unique smooth solution for any smooth function $f>0$,  with uniform $C^{1, 1}$ bound (independent of $\inf f$ in particular). 

\begin{thm}\label{mainthm1}Let $n\geq 4$. Given $u_0, u_1$ such that $g_{u_i}\in \cC^+, i=1, 2$, then there exists a unique smooth solution $u(t)$ of \eqref{geodesic1} such that $u(0, \cdot)=u_0$, $u(1, \cdot)=u_1$. Moreover, we have the following uniform $C^{1, 1}$ estimate,
\begin{equation}
|u|_{C^0}+|u_t|\leq C=C(C_2, \sup f),\; \max \{|\nabla u|, u_{tt}, |\nabla^2 u|, |\nabla u_t|\}\leq C_3.
\end{equation}
\end{thm}

\begin{rmk}We use the following convention of dependence of the constants. We use $C_1$ to denote a uniformly bounded (positive) constant depending only on $(M, g)$, $C_2$ to denote a uniformly bounded constant depending in addition on the boundary value $u_0, u_1$, and $C_3$ to denote a uniformly bounded constant depending in addition on $f$. An important feature  is that $C_3$ does not depend on $\inf f$, but rather on 
\[
\{\sup f +f^{-1}\left(|\nabla f|+|f_t|+|f_{tt}|+|\Delta f|\right)\}.
\] 
We also use the notation $C=C(a_1, a_2, \cdots )$ to denote a uniform constant which depends on parameters $a_1, a_2, \cdots$. 

\end{rmk}

\begin{rmk}The $C^1$ estimates and the boundary $C^2$ estimates have been obtained by Gursky-Streets \cite{GS2}. The essential new ingredient of our results is the interior $C^2$ estimates of \eqref{geodesic1}. The appearance of the nonlinear first order terms in the Schouten tensor $A_u$,  a ``nonstandard" nonlinearity (the operator $F$ is not symmetric) and the curvature of the background metric are the major causes of the difficulties. There are two major observations in our approach to solve \eqref{geodesic1}. The first is the concavity of the operator $G=\log F$ (for $k=2$), with
\[
F(u_{tt}, A_u, \nabla u_t)=u_{tt}\sigma_2(A_u)-\langle T_1(A_u), \nabla u_t\otimes \nabla u_t\rangle.
\]
For fully nonlinear elliptic equations, the concavity of the operator is essential. 
In particular this concavity is necessary for H\"older estimate of second order when applying the Evans-Krylov theory to obtain higher regularities. In \cite{GS2} Gursky-Streets quote the concavity of $\sigma_k^{\frac{1}{k}}(u_{tt}^{\frac{1-k}{k}}E_u)$ to apply Evans-Krylov theory. We believe this is not sufficient since the concavity of $\log F$ (or $F^{\frac{1}{3}}$) is not a direct consequence of the concavity of $\sigma_k^{\frac{1}{k}}$, due to the complicated nonlinearity (in terms of $D^2 u$) of \[u_{tt}^{\frac{1-k}{k}}E_u=u_{tt}^{\frac{1}{k}}A_u-u_{tt}^{\frac{1-k}{k}}\nabla u_t\otimes \nabla u_t.\] 
The concavity of $\log F$ also simplifies the computations greatly to derive interior $C^2$ estimates of $u_{tt}, \Delta u$, compared with [Section 4]\cite{GS2}. The second observation is that the appearance of the nonlinear first order terms $\nabla u \otimes \nabla u-|\nabla u|^2 g/2$ in $A_u$ will result in a quadratic form in the computations of $\cL_F(t_{tt})$ and $\cL_F(\Delta u)$. This quadratic form contains terms with high power of second order derivatives. Luckily, this quadratic form is positive definite when $n\geq 4$ and this is the sign in favor of applying the maximum principle. When $n=3$, this quadratic form contains ``bad terms" of high power (fourth power) of second order derivatives. It seems to be extremely hard to control them. Hence our approach only works for $n\geq 4$. 
When $n\geq 5$, the quadratic form is sufficiently positive which makes the argument of second order estimates quite straightforward. The case when $n=4$ is subtle and we refer readers to Section 3 for details. 
\end{rmk}

As a direct corollary, we have the following,
\begin{thm}\label{mainthm2}Let $n\geq 4$. Given $u_0, u_1$ such that $g_{u_i}\in \cC^+, i=1, 2$, there exists a $C^{1, 1}$ function $u(t)$ which solves \eqref{geodesic} in the strong sense, such that $u(0, \cdot)=u_0$, $u(1, \cdot)=u_1$. 
\end{thm}

\begin{rmk}
We believe the $C^{1, 1}$ solution is unique but we are not able to establish the uniqueness directly. The uniqueness of fully nonlinear degenerate elliptic equation can be a subtle problem. For geometric applications, we mainly use the approximating smooth solutions $u^s$ with a parameter $s\in (0, 1]$, which is smooth and approximates $u$ in a precise way when $s\rightarrow 0$.  The approximating solution would be sufficient for the applications. \end{rmk}

Given the $C^{1, 1}$ regularity, we can verify that the formal metric picture set up by Gursky-Streets holds strictly. In particular we prove the convexity of $\cF$-functional along the $C^{1, 1}$ geodesic $(n=4)$. The convexity of $\cF$ will give a straightforward argument of uniqueness of $\sigma_2$-Yamabe problem. In particular we have the following,
\begin{thm}Let $(M^4, g)$ be a compact four manifold such that $\cC^+\neq \emptyset$. Then $\cC^+$ is a metric space with Gursky-Streets' metric and it has nonpositive curvature in the sense of Alexanderov. Given $u_0, u_1\in \cC^+$ and let $u^s: [0, 1]\times M\rightarrow \cC^+$ be the approximating geodesic with the boundary datum $u_0, u_1$, satisfying, for $s\in (0, 1]$
\[
u^s_{tt}\sigma_2(A_{u^s})-\langle T_1(A_{u^s}), \nabla u^s_t\otimes \nabla u^s_t\rangle=s
\] 
 Let $u$ be the limit of $u^s$, which defines a $C^{1, 1}$ geodesic. Then $\cF$ is convex along the path $u$.
\end{thm}

With the convexity of $\cF$, we can prove that if $u_0, u_1$ are two critical points, then the path $u$ is either trivial ($\p_t u=\text{const}$) or $(M^4, g_u)$ is isometric to $S^4$ with the round metric. This  gives a direct proof of the uniqueness of $\sigma_2$-Yamabe problem,
 \begin{cor}Let $(M^4, g)$ be a compact four manifold with $\cC^+\neq \emptyset$. 
\begin{enumerate}\item There exists a unique solution to the $\sigma_2$-Yamabe problem in $[g]$ if $(M^4, g)$ is not conformally equivalent to the round $S^4$.

\item In $[g_{S^4}]$, all solutions to the $\sigma_2$-problem are round metrics.
\end{enumerate}
\end{cor}

\begin{rmk}The above uniqueness of $\sigma_2$-Yamabe problem was proved by Gursky-Streets \cite{GS2}. Gursky-Streets solved a version of perturbed geodesic equation
\[
(1+\epsilon)u_{tt}\sigma_2(A_u)-\langle T_1(A_u), \nabla u_t\otimes \nabla u_t\rangle=sf
\]
 and obtained a uniform $C^{0, 1}$ estimate of the solution $u^{\epsilon, s}$ (independent of $s, \epsilon$). The $C^{1, 1}$ estimates in \cite{GS2} depend on $\epsilon^{-1}$.  The lack of uniform $C^{1, 1}$-estimate is overcome by the parabolic smoothing through the Guan-Wang flow with initial datum $u^{\epsilon, s}$ (with uniform estimates depending only on $C^1$ of the initial datum). 
 A technical point is that the concavity of the fully nonlinear elliptic operator (such as $\log F$) is necessary to obtain the higher regularity of $u^{\epsilon, s}$. 
 \end{rmk}

The concavity of the Gursky-Streets operator is rather subtle. In the course of proving its concavity (for $k=2$), we find a new convexity for matrices in $\Gamma^+_2$. We believe this convexity is of its own interest and we state it as the following theorem.
Suppose $r$ is a $n\times n$ symmetric matrix in $\Gamma^+_2$, we define the following function on $(r, Y)\in \Gamma^+_2\times \R^{n}$
\begin{equation}\label{H-operator}
H(r, Y)=\frac{T_1(r)(Y, Y)}{\sigma_2(r)}=\frac{\p \log \sigma_2(r)}{\p r_{ij}}(Y, Y),
\end{equation}
We identify the matrix $T_1(r)$ with the linear transformation it generates, and it induces a quadratic form 
\[
T_1(r)(Y, Y)=Y^t T_1(r)Y
\]

\begin{thm}
The function $H$ is convex on $\Gamma^+_2\times \R^n$. As a consequence, the Gursky-Streets' operator $G=\log F$ is concave. 
\end{thm}

In general we can define, for $(r, Y)\in \Gamma^+_k\times \R^n$, for $1\leq k\leq n$ 
\[
H_k=\frac{T_{k-1}(r)(Y, Y)}{\sigma_k(r)}=\frac{\p \log \sigma_k(r)}{\p r_{ij}}(Y, Y),
\]

We conjecture that for any $3\leq k\leq n-1$, $H_k$ is a convex function on $(r, Y)$ (note that when $k=n$, it is an old result of Marcus \cite{Marcus}). This would prove that the Gursky-Streets operator $\log F_k$ is concave. \\

{\bf Acknowledgement:} The author is supported in part by an NSF grant, no. 1611797.

\numberwithin{equation}{section}
\numberwithin{thm}{section}

\section{Preliminary}
In this section we recall Gursky-Streets' geodesic equation and related notations briefly.  
Let $(M^n, g)$ be a compact Riemannian manifold with the conformal class $[g]$. We write $Ric$ for Ricci tensor and $A$ for Schouten tensor. 
The metrics in $[g]$ can be parametrized by metrics of the form $g_u=e^{-2u}g$. The Ricci curvature is given by
\[
Ric(g_u)=Ric+(n-2)\left(\nabla^2 u+\nabla u\otimes \nabla u-\frac{1}{2}|\nabla u|^2g\right)+\left(\Delta u-\frac{n-2}{2}|\nabla u|^2\right) g
\]
and the scalar curvature is given by
\[
R(g_u)=e^{2u}\left(R+2(n-1)\left(\Delta u-\frac{n-2}{2}|\nabla u|^2\right)\right)
\]
Under the conformal change, the Schouten tensor is given by
\[
A_u=A(g_u)=A+\nabla^2 u+\nabla u\otimes \nabla u-\frac{1}{2}|\nabla u|^2 g.
\]
For $1\leq k\leq n$, $A\in\Gamma^+_k$ if $\sigma_j(g^{-1}A)>0$ for all $1\leq j\leq k$. An important case is when $n=4, k=2$. 
Denote
\[
\mathcal{C}^{+}=\mathcal{C}^{+}[g]=\{g_u=e^{-2u}g: A_u\in \Gamma^+_{2}\}.
\]
Gursky-Streets \cite{GS2} defined a metric on the space of $\mathcal{C}^{+}$ by,
\[
\langle \psi, \phi\rangle_{u}=\int_M \phi \psi \sigma_2(g_u^{-1}A_u) dV_u
\]
A direct computation \cite{Reilly} gives 
\begin{equation}\label{p101}
\frac{\p}{\p t}\left[\sigma_k(g_u^{-1}A_u)dV_u\right]=\langle T_{k-1}(g_u^{-1}A_u), \nabla^2_u \frac{\p u}{\p t}  \rangle_{g_u} dV_u+(n-2k)\frac{\p u}{\p t}\sigma_k(g_u^{-1}A_u)dV_u,
\end{equation}
where $T_{k-1}(g_u^{-1}A_u)$ is the $(k-1)$-th Newton transformation. Note that $T_{k-1}$ is a $(1, 1)$ tensor. In the paring in \eqref{p101}, we view $T_{k-1}$ as the corresponding $(0, 2)$ tensor using the metric $g_u$ to lower the index. In particular we have for $n=4, k=2$,
\[
T_1(g^{-1}_uA_u)=\sigma_1(g_u^{-1}A_u)g_u-A_u. 
\]

{\bf Convention:}
We use the convention as in \cite{GS2}, that we write $\sigma_2(A_u)=\sigma_2(g^{-1}A_u)$ and use the notation $\sigma_2(g_u^{-1}A_u)$ when we use $g_u$ to raise index. Of course these notations differ by a conformal factor. For example,
\[
\sigma_2(g_u^{-1}A_u)=e^{4u} \sigma_2(g^{-1}A_u)=e^{4u}\sigma_2(A_u). 
\]
Similarly we write $T_1(A_u)=T_1(g^{-1}A_u)$. 

Now let $u: [0, 1]\times M\rightarrow \mathbb{R}$ be a path in $\mathcal{C}^{+}$ (identifying $u$ with $g_u$), then the geodesic equation of Gursky-Streets metric is given by
\[
u_{tt}\sigma_2(A_u)=\langle T_1(A_u), \nabla u_t\otimes \nabla u_t\rangle,
\]
A key property is the following,
\begin{lemma}[Viaclovsky \cite{Via}]\label{free}For $k=2$ or if the manifold is locally conformally flat, $T_{k-1}(g^{-1}A)$ is  divergence free. 
\end{lemma}

We need some facts about the convex cone $\Gamma^+_k$ and the Newton transformation $T_k(A)$. With the standard Euclidean metric, the $k$-th Newton transformation associated with a symmetric matrix $S$ (on $\R^n$) is given by
\[
T_k(S)=\sigma_k(S)I-\sigma_{k-1}(S) S+\cdots+(-1)^kS^k.
\]
In particular $T_1(S)=\sigma_1(S)I-S$. 
\begin{prop}\label{gamma1}We have,
\begin{enumerate}
\item Each $\Gamma^+_k$ is an open convex cone.
\item If $A\in \Gamma^+_k$, then $T_{k-1}(A)$ is positive definite.
\item $\log \sigma_k$ and $\sigma_k^{1/k}$ are concave on $\Gamma^+_k$. 
\end{enumerate}
\end{prop}

We also need the following,
\begin{prop}\label{gamma2}\cite{GS2}Given $A$ a symmetric matrix and $X$ a vector, then
\begin{equation}
\begin{split}
&\langle T_k(A-X\otimes X), X\otimes X\rangle =\langle T_k(A), X\otimes X\rangle\\
&\sigma_k(A-X\otimes X)=\sigma_k(A)-\langle T_{k-1}(A), X\otimes X\rangle
\end{split}
\end{equation}
\end{prop}

Following \cite{GS2}, we denote $E_u=u_{tt}A_u-\nabla u_t\otimes \nabla u_t$. An important observation in \cite{GS2} is to rewrite the geodesic equation  as (using Proposition \ref{gamma2})
\[
\sigma_2(E_u)=u_{tt} \left(u_{tt}\sigma_2(A_u)-\langle T_1(A_u), \nabla u_t\otimes \nabla u_t\rangle 
\right)=0
\]

\begin{rmk}When $k=1$, $\sigma_1(E_u)=u_{tt}\sigma_1(A_u)-|\nabla u_t|^2$.  If we consider only the leading term in $A_u$, that is $\nabla^2u$, then $\sigma_1(E_u)=u_{tt}\Delta u-|\nabla u_t|^2$. This operator is introduced by S. Donaldson \cite{Donaldson} when he set up a formal Riemannian metric for the space of volume forms. The Donaldson operator can be viewed as a special case of Gursky-Streets operator. See Appendix for more discussions. 
\end{rmk}

{\bf Convention}: Given a symmetric matrix $A$, we need to diagonalize $A$ at times. Unless specified otherwise, we use the convention that $\lambda_1\geq \cdots \geq \lambda_n$, where $\lambda_i$ are the eigenvalues of $A$. And we use $\sigma_1, \cdots, \sigma_n$ to denote the basic symmetric functions of $\lambda_1, \cdots, \lambda_n$ if there is no confusion. In general we use $\sigma_i(a_1, \cdots, a_k)$ to denote the $i$-th symmetric function of $(a_1, \cdots, a_k)$ for $i\leq k$, and it is zero when $i>k$. 

\begin{prop}\label{gamma}Suppose $A_u\in \Gamma_2^{+}$ and $u_{tt}>0, \sigma_2(E_u)>0$, then $E_u\in \Gamma^{+}_2$. In particular,
we have
\begin{equation}\label{gamma3}
\sigma_1(E_u)\geq f\sigma_2 (A_u)^{-1}\sigma_1(A_u)
\end{equation}
if we write $u_{tt}^{-1}\sigma_2(E_u) =f$. 
\end{prop}
\begin{proof}We only need to show that $\sigma_1(E_u)>0$, that is $u_{tt}\sigma_1(A_u)-|\nabla u_t|^2>0$. Since we have 
\[
u_{tt}^{-1}\sigma_2(E_u)=u_{tt}\sigma_2(A_u)-\langle T_1(A_u), \nabla u_t\otimes \nabla u_t\rangle =f>0,
\]
it follows that
\[
u_{tt}=f\sigma_2(A_u)^{-1}+\sigma_2(A_u)^{-1}\langle T_1(A_u), \nabla u_t\otimes \nabla u_t\rangle.
\]
We compute
\[
\sigma_1(E_u)=f\sigma_2(A_u)^{-1}\sigma_1(A_u)+\frac{\sigma_1(A_u)}{\sigma_2(A_u)}\langle T_1(A_u), \nabla u_t\otimes \nabla u_t\rangle-|\nabla u_t|^2
\]
It is sufficient to argue that,
\[
\frac{\sigma_1(A_u)}{\sigma_2(A_u)}\langle T_1(A_u), \nabla u_t\otimes \nabla u_t\rangle\geq |\nabla u_t|^2
\]
We claim that 
\[
\frac{\sigma_1(A_u)}{\sigma_2(A_u)}T_1(A_u)\geq I.
\]
Diagonalize $A_u$ with eigenvalues $\lambda_1\geq \cdots \geq \lambda_n$, we need to verify that for each $i$ (or $i=1$), 
\[
\sigma_1(\sigma_1-\lambda_i)\geq \sigma_2
\]
This is to show that
\[\left(\sum_{i\neq 1} \lambda_i\right)^2\geq \sigma_2(\lambda_2, \cdots, \lambda_n)\]
This is obvious. 
\end{proof}

Denote the operator
\begin{equation}\label{F}
F(u_{tt}, A_u, \nabla u_t):=u_{tt}^{-1}\sigma_2(E_u)=u_{tt}\sigma_2(A_u)-\langle T_1(A_u), \nabla u_t\otimes \nabla u_t\rangle.
\end{equation}
We want to solve the Dirichlet problem, with $u(0, \cdot)=u_0, u(1, \cdot)=u_1$
\begin{equation}\label{E}
F(u_{tt}, A_u, \nabla u_t)=sf
\end{equation}
for $s\in (0, 1]$ and a positive smooth function $f$. The main point is to derive a uniform $C^{1, 1}$ estimate, independent of $s$. For simplicity of notation, we will derive the a prior estimates for the equation of the form, with a general right hand side, 
\begin{equation}\label{E2}
F(u_{tt}, A_u, \nabla u_t)=f.
\end{equation}

\begin{prop}Given $u\in C^2$ such that $A_u\in \Gamma^+_2$, then the equation \eqref{E2} is strictly elliptic when $f>0$.
The linearized operator is given by
\begin{equation}\label{E3}
\begin{split}
\cL_F(v)=&u_{tt}^{-1}\langle T_1(E_u), v_{tt} A_u+u_{tt}\cL_{A_u}(v)-\nabla u_t\otimes \nabla v_t-\nabla v_t\otimes \nabla u_t\rangle -u_{tt}^{-2}\sigma_2(E_u)v_{tt}\\
=&u_{tt}^{-1} f v_{tt}+u_{tt}^{-1}\langle T_1(E_u), u_{tt}\cL_{A_u}(v)-\nabla u_t\otimes \nabla v_t-\nabla v_t\otimes \nabla u_t+u_{tt}^{-1}v_{tt} \nabla u_t\otimes \nabla u_t\rangle
\end{split}
\end{equation}
where $\cL_{A_u}(v)$ is the linearization of $A_u$, given by 
\[
\cL_{A_u}(v)=\nabla^2 v+\nabla u\otimes \nabla v+\nabla v\otimes \nabla u-\langle \nabla u, \nabla v\rangle g.\]
\end{prop}
\begin{proof}
First note that when $f>0$, by the assumption $A_u\in \Gamma^+_2$, $u_{tt}>0$. 
Suppose $\delta u=v$, and we use the variation of $\sigma_2$, 
$\delta \sigma_2(E_u)=\langle T_1(E_u), \delta E_u\rangle$. By direct computation we have
\[
\cL_F(v)=u_{tt}^{-1}\langle T_1(E_u), v_{tt} A_u+u_{tt}\cL_{A_u}(v)-\nabla u_t\otimes \nabla v_t-\nabla v_t\otimes \nabla u_t\rangle -u_{tt}^{-2}\sigma_2(E_u)v_{tt}.\]
To show the ellipticity, we only need to take care of second order derivatives of $v$. The leading terms reads,
\[
u_{tt}^{-1}\langle T_1(E_u), v_{tt} A_u+u_{tt} \nabla^2 v-\nabla u_t\otimes \nabla v_t-\nabla v_t\otimes \nabla u_t\rangle -u_{tt}^{-2}\sigma_2(E_u) v_{tt}
\]
Replacing the derivatives of $(v_t, \nabla v)$ by a vector $(\xi, X)\in T([0, 1]\times M)=\R\times \R^{n}$, we need to show that the following quadratic form is positive definite,
\[
Q(\xi, X):=\langle T_1(E_u), \xi^2 A_u+u_{tt} X\otimes X-\xi \nabla u_t\otimes X-\xi X\otimes \nabla u_t\rangle -u_{tt}^{-1}\sigma_2(E_u) \xi^2
\]
We compute
\[
\begin{split}
\xi^2 A_u+u_{tt} X\otimes X-\xi \nabla u_t\otimes X-\xi X\otimes \nabla u_t=&\xi^2(A_u-u_{tt}^{-1}\nabla u_t\otimes \nabla u_t)+Y\otimes Y\\
=&u^{-1}_{tt} \xi^2 E_u+Y\otimes Y
\end{split}
\]
where $Y=\sqrt{u_{tt}}X-\xi \nabla u_t$. It follows that
\[
Q(\xi, X)=\langle T_1(E_u), Y\otimes Y\rangle+ u^{-1}_{tt}\left(\langle T_1(E_u), E_u\rangle-\sigma_2(E_u)\right) \xi^2. 
\]
Since $E_u\in \Gamma^{+}_2$, $T_1(E_u)>0$. A direct computation gives 
\begin{equation}\label{Eu1}
\langle T_1(E_u),  E_u\rangle-\sigma_2(E_u)=\sigma_2(E_u)>0
\end{equation}
It then follows that, for $(\xi, X)\neq 0$, 
$Q(\xi, X)>0$. 
To show the second identity in \eqref{E3}, we compute
\[
\langle T_1(E_u), v_{tt} A_u\rangle= u_{tt}^{-1}\langle T_1(E_u), E_u\rangle v_{tt}+\langle T_1(E_u), u_{tt}^{-1} v_{tt} \nabla u_t\otimes \nabla u_t\rangle. 
\]
Applying \eqref{Eu1} again we get the result. 
This completes the proof. 
\end{proof}

The following concavity of $F$ is essential for us and this would be proved in the appendix. Denote $r$ to be a symmetric $n\times n$ matrix such that $r\in \Gamma^+_2$, and $R$ to be a $(n+1)\times (n+1)$ matrix with $Y=(y_1, \cdots, y_n)\in \R^n$, 
\[
R=\begin{pmatrix} x & Y\\
Y^t& r
\end{pmatrix}
\]
\begin{lemma}\label{concavity}The function
\[
G(R)=\log \left(x\sigma_2(r)-Y^tT_1(r)Y\right)
\]
is concave on $R$ for  $r\in \Gamma^+_2$ such that $x\sigma_2(r)-Y^tT_1(r)Y>0$. In particular $\log F=\log F(u_{tt}, A_u, \nabla u_t)$ is a concave elliptic operator. 
\end{lemma}

\section{A priori estimates}
In  this section we derive the \emph{a priori estimates} to solve the equation.
Given $u_0, u_1\in C^\infty$ such that $A_{u_0}, A_{u_1}\in \Gamma^+_2$,  we assume that $u\in C^\infty$ such that $A_u\in \Gamma^+_2$, and solves the equation $F(u_{tt}, A_u, \nabla u_t)=f$, for a positive function $f\in C^\infty$, with the boundary condition $u(0, x)=u_0(x)$, $u_1(1, x)=u_1(x)$.  
We will need the comparison function as follows. Denote $U_a=at(1-t)+(1-t)u_0+t u_1$ for any number $a$. Note that $U_0=u_0$ at $t=0$, $U_1=u_1$ at $t=1$ for any $t$. In particular $U_a$ has the same boundary value with $u$. 

\begin{defn}A smooth function $u$ is called admissible if $A_u\in \Gamma^+_2$.
\end{defn}

Moreover, since $u_0, u_1$ are admissible ($A_{u_i}\in \Gamma^+_2$ for $i=0, 1$), $U_0=(1-t)u_0+t u_1$ is admissible \cite{Via} and hence $U_a$ is all admissible for any $a$. In particular $A_{U_0}=A_{U_a}, (\nabla U_0)_t=(\nabla U_{a})_t$ for any $a$. 
Gursky-Streets \cite{GS2} proved a uniform $C^1$ estimate for the equation
\[
u_{tt}^{1-k}\sigma_k(E_u^\epsilon)=f,
\]
where $E_u^\epsilon=(1+\epsilon)u_{tt}A_u-\nabla u_t\otimes \nabla u_t$, for any $k\geq 1$. They introduced an extra $\epsilon$-parameter for the purpose of $C^2$ estimates, which do not play any essential role in $C^1$ estimates. Hence their results clearly apply in our setting to obtain uniform $C^1$ estimates. In particular most computations required in $C^1$ estimates can be found in \cite{GS2}. Nevertheless we will include details of $C^1$ estimates for completeness. 
The main reason is that these computations will be needed for uniform $C^2$ estimates. 

\subsection{$C^0$ estimates}
In this section we derive the $C^0$ estimates. We use the concavity of $\log F$ in a significant way and our $C^0$ estimate makes the bound on $u_t$ straightforward. Moreover our estimates are slightly sharper at times using the concavity of $G=\log F$. 

\begin{prop}\label{c0}There exists $a=a(u_0, u_1, \sup f)>0$ sufficiently large, such that 
\begin{equation*}
U_{-a}\leq u\leq U_0=(1-t)u_0+t u_1.
\end{equation*}
\end{prop}

\begin{proof}First by $u_{tt}>0$, we have
\[
\frac{u(\cdot, t)-u(\cdot, 0)}{t-0}<\frac{u(\cdot, 1)-u(\cdot, t)}{1-t}
\]
That gives the upper bound, 
\[
u(\cdot, t)<(1-t)u(\cdot, 0)+tu(\cdot, 1)=(1-t)u_0+t u_1.
\]
We claim $u-U_{-a}\geq 0$ for $a>0$ sufficiently large. We argue by contradiction. 
Since $u-U_{-a}=0$ for $t=0$ and $t=1$,  there exists an interior point $p=(t, x)\in (0, 1)\times M$, such that $u-U_{-a}$ obtains its minimum at $p$. Denote $u^s=su-(1-s)U_{-a}$ and $v=\p_s u^s(s=1)=u-U_{-a}$. Then $D^2v\geq 0$  and $\nabla v=0$ at $p$. By the concavity of $\log F$, it follows that for $s\in [0, 1]$,
\[
\log F(u^s_{tt}, A_{u^s}, \nabla u^s_t)\geq s\log F(u_{tt}, A_u, \nabla u_t)+(1-s)\log F({(U_{-a})}_{tt}, A_{U_{-a}}, \nabla {(U_{-a})}_{t})
\]
At $s=1$, we get (at $p$), 
\begin{equation}\label{zero1}
F^{-1}\cL_F (v)\leq \log F(u_{tt}, A_u, \nabla u_t)-\log F({(U_{-a})}_{tt}, A_{U_{-a}}, \nabla {(U_{-a})}_{t}),
\end{equation}
where $F^{-1}\cL_F$ takes values at $u\; (s=1)$. 
We can choose $a$ large enough such that $$F({(U_{-a})}_{tt}, A_{U_{-a}}, \nabla {(U_{-a})}_{t})=2a \sigma_2(A_{U_{0}})-(T_1(A_{U_{0}}), \nabla {(U_{0})}_{t}\otimes \nabla {(U_{0})}_{t})$$ is sufficiently large. Then the right hand side of \eqref{zero1} is negative (at $p$) since $ F(u_{tt}, A_u, \nabla u_t)=f$.  This is a contradiction given the claim that $\cL_F(v)\geq 0$ at $p$. Note that $D^2 v\geq 0, \nabla v=0$ at $p$. We compute, using $\nabla v=0$ at $p$,
\[
\cL_F(v)=u_{tt}^{-1}\langle T_1(E_u), v_{tt} A_u+u_{tt} \nabla^2 v-\nabla u_t\otimes \nabla v_t-\nabla v_t\otimes \nabla u_t\rangle -u_{tt}^{-2}\sigma_2(E_u) v_{tt}
\]
We can assume $v_{tt}>0$. Otherwise we have $v_{tt}=0$, then $\nabla v_t=0$ and $\nabla^2 v\geq 0$ (since $D^2 v\geq 0$ at $p$). In this case the claim follows trivially. If $v_{tt}>0$, the argument follows similarly as in Proposition 2.5. Indeed we write
\[
\cL_F(v)=u_{tt}^{-1}\langle T_1(E_u), v_{tt} (A_u-u_{tt}^{-1}\nabla u_t\otimes \nabla u_t)+Y\otimes Y+u_{tt}(\nabla^2v-v_{tt}^{-1}\nabla v_t\otimes \nabla v_t)\rangle -u_{tt}^{-2}\sigma_2(E_u) v_{tt},
\]
where $Y=\sqrt{v_{tt}/u_{tt}} \nabla u_t-\sqrt{u_{tt}/v_{tt}}\nabla v_t$. By \eqref{Eu1} and the positivity of $\nabla^2v-v_{tt}^{-1}\nabla v_t\otimes \nabla v_t$ (this is because $D^2 v\geq 0$), it follows that $\cL_F(v)\geq 0$. 
\end{proof}

\subsection{$C^1$ estimates}First we have the following,

\begin{prop}\label{c1t}Let $a$ be the constant in Proposition \ref{c0}. Then we have, 
\[
-a+u_1-u_0\leq u_t\leq a+u_1-u_0
\]
\end{prop}
\begin{proof}
Since $u_{tt}>0$, it follows that $u_t(t, x)$ is increasing in $t$. Hence we only need to argue $u_t(0, x)\leq u_t(1, x)$ are both bounded. We compute, using Proposition \ref{c0},
\[
u_t(0, \cdot)=\lim_{t\rightarrow 0}\frac{u(t, \cdot)-u(0, \cdot)}{t}\geq \lim_{t\rightarrow 0} \frac{at(t-1)+t(u_1-u_0)}{t}=-a+u_1-u_0. 
\]
It is evident that $u_t(0, \cdot)\leq u_1-u_0$ by convexity. Similarly we have $u_1-u_0\leq u_t(1, \cdot)\leq a+u_1-u_0$. 

\end{proof}

To derive estimates of $|\nabla u|^2$ and second order derivatives, we need some preparation due to the complicated computations. 
First we need to choose a normalization condition. Note that if $u$ is admissible, then $\tilde u=u-c_1 t-c_2$ is also admissible since $A_u, E_u$ do not change at all. In particular if $u$ is a solution, then $\tilde u=u-c_1t-c_2$ is also a solution since  $\nabla \tilde u=\nabla u, D^2\tilde u=D^2 u$. The corresponding boundary condition is changed by a constant with $\tilde u_0=u_0-c_2, \tilde u_1=u_1-c_1-c_2$ and  $\tilde u_t=u_t-c_1$. Hence we can choose two sufficiently large constants $c_1$ and $c_2$ such that $\tilde u\leq -1$, and $\tilde u_t\leq -1$.  From now on  we choose such a normalization condition on $u_0, u_1$ such that, 
\begin{equation}\label{normalization}
-c_0\leq u\leq -1, -c_0\leq u_t\leq -1,
\end{equation} 
where $c_0$ is the uniform bound we have obtained for $|u|$ and $|u_t|$. \\

Next we compute $\cL_F(v)$ for various barrier functions $v$. The philosophy is well-known in nonlinear elliptic theory, to construct various barrier functions $v$ such that
\[
\cL_F(v)\geq -C+\text{good positive terms}
\]
Such barrier functions serve as the purpose of \emph{subharmonic functions} (or \emph{subsolutions}) with respect to $\cL_F$ and play an essential role in the maximum principle argument. 
The first such function is the $t$-functions,
\begin{prop}
Suppose $v=v(t)$ is a $t$-function, then
\begin{equation}
\begin{split}
\cL_F(v)=&\frac{v_{tt}}{u^2_{tt}}\left(\langle T_1(E_u), u_{tt}A_u\rangle-\sigma_2(E_u)\right)\\
=&\frac{v_{tt}}{u_{tt}^2}\left(\langle T_1(E_u), E_u+\nabla u_t\otimes \nabla u_t\rangle-\sigma_2(E_u)\right)\\
=&\frac{v_{tt}}{u_{tt}^{2}}\left(\sigma_2(E_u)+\langle T_1(E_u), \nabla u_t\otimes \nabla u_t\rangle \right)\\
=&v_{tt}\sigma_2(A_u),
\end{split}
\end{equation} 
where we apply Proposition \ref{gamma2} in the last step above.
In particular, \begin{equation}\label{t}
\cL_F(t^2)=2\sigma_2(A_u)\end{equation}
\end{prop}

The second choice is the function $-u$ itself. 
We compute $\cL_F(u)$.
\begin{prop}\label{u}We have,
\begin{equation}\label{u1}
\cL_F(u)=3u_{tt}^{-1}\sigma_2(E_u)+\langle T_1(E_u), -A+\nabla u\otimes \nabla u-\frac{1}{2}|\nabla u|^2 g\rangle
\end{equation}
\end{prop}
\begin{proof}By \eqref{E3}, we compute
\[
\begin{split}
\cL_F(u)=&u_{tt}^{-1}\langle T_1(E_u), u_{tt}A_u+u_{tt}\cL_{A_u}(u)-2\nabla u_t\otimes \nabla u_t\rangle-u_{tt}^{-1}\sigma_2(E_u)\\
=&u_{tt}^{-1}\langle T_1(E_u), 2E_u-u_{tt}A +u_{tt}\nabla u\otimes \nabla u-\frac{u_{tt}}{2}|\nabla u|^2 g \rangle-u_{tt}^{-1}\sigma_2(E_u)\\
=&3u_{tt}^{-1}\sigma_2(E_u)+\langle T_1(E_u), -A +\nabla u\otimes \nabla u-\frac{1}{2}|\nabla u|^2 g \rangle.
\end{split}
\]
where we have used \eqref{Eu1}.
\end{proof}

\begin{rmk}Both the propositions above are derived in \cite{GS2} for general $k$. We include the computations here for completeness. 
\end{rmk}

We use the operator $D=(\p_t, \nabla)$ to denote the gradient on $\R\times M$, where the space derivative $\phi_k$ denotes the covariant derivative $\nabla _k\phi$. We rewrite \eqref{E3} as,
\begin{equation}\label{E4}
\cL_F(\phi)= u_{tt}^{-1}f \phi_{tt}+u_{tt}^{-1}\left\langle T_1(E_u), P_u(D^2\phi)\right \rangle,
\end{equation}
where 
\begin{equation}\label{P1}
P_u(D^2\phi)=u_{tt}\cL_{A_u} \phi-\nabla u_t\otimes \nabla \phi_t-\nabla \phi_t\otimes \nabla u_t+u_{tt}^{-1}\phi_{tt} \nabla u_t\otimes \nabla u_t.
\end{equation}
For simplicity we denote the symmetric tensor product as follows, 
\[
X\boxtimes Y:=X\otimes Y+Y\otimes X
\]
\begin{prop}\label{quadratic}We have the following,
\begin{equation}\label{E5}
\cL_F(\phi\psi)=\phi \cL_F(\psi)+\psi \cL_F(\phi)+Q_u(D\phi, D\psi)+2u_{tt}^{-1}f \phi_t\psi_t,
\end{equation}
where $Q_u$ is a quadratic form on $D\phi, D\psi$ given by
\begin{equation*}
Q_u(D\phi, D\psi)=u_{tt}^{-1}\left\langle T_1(E_u), u_{tt}\nabla \phi\boxtimes \nabla \psi-\phi_t \nabla u_t\boxtimes \nabla \psi-\psi_t \nabla u_t\boxtimes \nabla \phi+2u_{tt}^{-1}\phi_t\psi_t \nabla u_t\otimes \nabla u_t\right \rangle
\end{equation*}
Moreover, we compute
\begin{equation}\label{exponential}\cL_F(e^\phi)=e^\phi \cL_F(\phi)+e^{\phi}\left(\frac{1}{2}Q_u(D\phi, D\phi)+u_{tt}^{-1} f\phi_t^2\right)
\end{equation}
An important feature is that $Q_u$ is positive definite in the sense that
\[
Q_u(D\phi, D\phi)\geq 0. 
\]
\end{prop}
\begin{proof}
This is a straightforward computation. The main point is that $\cL_F$ and $P_u$ are second order linear differential operator and the product rule would introduce mixed terms on first derivatives, which lead to the terms $Q_u(D\phi, D\psi)+2u_{tt}^{-1}f \phi_t\psi_t$. Similarly this applies to $e^\phi$. Since $E_u\in \Gamma^+_2$, $T_1(E_u)>0$, it follows that 
\[
Q_u(D\phi, D\phi)=2u_{tt}^{-1}\langle T_1(E_u), Y\otimes Y\rangle\geq 0,
\]
where  $Y=(\sqrt{u_{tt}}) \nabla \phi-\phi_t \nabla u_t(\sqrt{u_{tt}})^{-1} .$ Clearly the positivity of $Q$ is simply the consequence of the ellipticity of $F$. 
\end{proof}

\begin{prop}We compute, using  Proposition \ref{u} and Proposition \ref{quadratic},
\begin{equation}\label{g2}
\cL_F(e^{-\l u})=\l e^{-\l u} \cL_F(-u)+\l^2e^{-\l u}\left(\frac{1}{2}Q_u(Du, Du)+u_{tt}^{-1}f u_t^2\right).
\end{equation}
\end{prop}

\begin{prop}We compute,
\begin{equation}\label{E9}
\cL_F(u_t^2)=2u_t f_t+2fu_{tt}
\end{equation}
\end{prop}
\begin{proof}By \eqref{E5}, we have
\[
\cL_F(u_t^2)=2u_t\cL_F(u_t)+Q_u(Du_t, D u_t)+2fu_{tt}
\]
Since taking time derivative has the same effect of taking variation, this gives
\[
\cL_F(u_t)=\p_t F=f_t. 
\]
It is clear that $Q_u(Du_t, Du_t)=0$. This completes the computation. 
\end{proof}

\begin{prop}\label{gradient}We compute
\begin{equation}\label{C1}
\cL_F(|\nabla u|^2)=2\nabla f \nabla u-2\left\langle T_1(E_u), \nabla u\nabla A+Rm(\nabla u, \nabla u)\right\rangle+Q_u(Du_i, Du_i)+2u_{tt}^{-1}f|\nabla u_t|^2,
\end{equation}
where we denote, 
\[
Rm(\nabla u, \nabla u)=R_{ilpk} u_i u_p \p_l\otimes \p_k
\]
\end{prop}
\begin{proof}First we compute, applying \eqref{E5} to $\phi=\psi=u_i$, 
\begin{equation}\label{E10}
\cL_F(|\nabla u|^2)=2u_i\cL_F(u_i)+Q_u(Du_i, Du_i)+2u_{tt}^{-1}f|\nabla u_t|^2. 
\end{equation}
Now we compute
\[
u_{tt}\cL_F(u_i)= f u_{itt}+\left\langle T_1(E_u), u_{tt}(\nabla^2 u_i+\nabla u\boxtimes \nabla u_i-(\nabla u, \nabla u_i)g)-\nabla u_t\boxtimes \nabla u_{it}+u_{tt}^{-1} u_{itt}\nabla u_t\otimes \nabla u_t\right\rangle
\]
Taking derivative of $\sigma_2(E_u)=fu_{tt}$, we get
\[
\left\langle T_1(E_u), \nabla_i E_u\right\rangle=f_i u_{tt}+fu_{tti}
\]
We compute
\[
\nabla_i E_u=u_{tti}A_u+u_{tt}\left(\nabla_i A+\nabla_i\nabla^2 u+\nabla_i\nabla u\boxtimes \nabla u-(\nabla_i \nabla u, \nabla u) g\right)-\nabla u_{ti}\boxtimes \nabla u_t
\]
Note that 
\[
\nabla_i \nabla^2 u-\nabla^2 \nabla_i u=u_{tt}R_{ilpk} u_p \p_l\otimes \p_k
\]
It follows that
\begin{equation}
\begin{split}
u_{tt}\cL_F(u_i)=&f u_{itt}+\left\langle T_1(E_u), \nabla_i E_u-u_{tti}A_u-u_{tt}(\nabla_i A+R_{ilpk} u_p \p_l\otimes \p_k)+\frac{u_{itt}}{u_{tt}}\nabla u_t\otimes \nabla u_t\right\rangle\\
=& 2fu_{tti}+f_i u_{tt}-\left\langle T_1(E_u), u_{tti} u_{tt}^{-1} E_u+u_{tt}\nabla_i A+u_{tt}R_{ilpk} u_p \p_l\otimes \p_k\right\rangle\\
=&f_i u_{tt}-u_{tt}\left\langle T_1(E_u), \nabla_i A+R_{ilpk} u_p \p_l\otimes \p_k\right\rangle
\end{split}
\end{equation}
Hence we have
\begin{equation}\label{E6}
\cL_F(u_i)=f_i-\left\langle T_1(E_u), \nabla_i A+R_{ilpk} u_p \p_l\otimes \p_k\right\rangle
\end{equation}
This completes the computation by combining \eqref{E10} and \eqref{E6}. 
\end{proof}

\begin{rmk}The computations above are essentially derived in \cite{GS2} for general $k$. We use the quadratic form $Q_u$ to simplify the notations and computations. Of course the positivity of $Q_u$ is essentially equivalent to the fact that $F$ is an elliptic operator. 
\end{rmk}

Now we prove the estimate for $|\nabla u|^2$.  Since $T_1(A)>0$ then there exists $c_0$ such that $T_1(A)\geq c_0 g$.  In particular for any $E\in \Gamma^+_2$, we assume there exists a uniformly positive constant $c_1$ such that,
\[
\langle T_1(E), A\rangle=\langle E, T_1(A)\rangle \geq c_0 \sigma_1(E)=(n-1)c_1 \sigma_1(E)=c_1 \sigma_1(T_1(E))
\]
When there is no confusion, we also write $\sigma_1(T_1)=\sigma_1(T_1(E)).$
Combining all the computations above, we have the following estimates,
\begin{lemma}\label{keylemma}
For $\l, b\geq 1$ sufficiently large, we have
\begin{equation}
\cL_F(e^{-\l u}+bt^2)\geq -C_4 f+e^\l \sigma_1(T_1) (1+|\nabla u|^2)+e^\l(\sigma_2(A_u)+f u_t^2 u_{tt}^{-1}).
\end{equation}
where $C_4=C_4(\l, |u|_{C^0})$. 
\end{lemma}

\begin{proof}
By Proposition \ref{u1}, we get
\[
\begin{split}
\cL_F(-u)=&-3f+\langle T_1(E_u), A-\nabla u\otimes \nabla u+\frac{1}{2} |\nabla u|^2 g\rangle\\
\geq &-3f+c_1 \sigma_1(T_1)+\frac{1}{2} \sigma_1(T_1)|\nabla u|^2-\langle  T_1(E_u), \nabla u\otimes \nabla u\rangle.
\end{split}
\]
We claim that for a constant $C_2\geq 2 u_t^2$, 
\[
Q_u(Du, Du)+C_2\sigma_2(A_u)\geq \langle T_1(E_u), \nabla u\otimes \nabla u\rangle
\]
We estimate,
\[
\begin{split}
Q_u(Du, Du)=&\frac{2}{u_{tt}}\langle T_1(E_u), u_{tt}\nabla u\otimes \nabla u-u_t\nabla u\boxtimes \nabla u_t+u_{tt}^{-1}u_t^2 \nabla u_t\otimes \nabla u_t\rangle\\
= &u_{tt}^{-1}\langle T_1(E_u), u_{tt}\nabla u\otimes \nabla u-2u_t\nabla u\boxtimes \nabla u_t+4u_{tt}^{-1}u_t^2 \nabla u_t\otimes \nabla u_t\rangle\\
&+\langle T_1(E_u), \nabla u\otimes \nabla u\rangle-2u_{tt}^{-2} u_t^2\langle T_1(E_u), \nabla u_t\otimes \nabla u_t\rangle\\
=&u_{tt}^{-1}\langle T_1(E_u), Y\otimes Y\rangle+\langle T_1(E_u), \nabla u\otimes \nabla u\rangle-2u_{tt}^{-2} u_t^2\langle T_1(E_u), \nabla u_t\otimes \nabla u_t\rangle,
\end{split}
\]
where $Y=\sqrt{u_{tt}}\nabla u-2(\sqrt{u_{tt}})^{-1} u_t \nabla u_t$. The claim follows since
\[
\sigma_2(A_u)-u_{tt}^{-2} \langle T_1(E_u), \nabla u_t\otimes \nabla u_t\rangle= f u_{tt}^{-1}\geq 0. 
\]
Choose $b\geq (C_2+1) \l^2 e^{-\l u}$, then we estimate
\begin{equation}
\begin{split}
\cL_F(e^{-\l u}+bt^2)\geq & \l e^{-\l u}\left(-3f+c_1 \sigma_1(T_1)+\frac{1}{2} \sigma_1(T_1)|\nabla u|^2-\langle  T_1(E_u), \nabla u\otimes \nabla u\rangle\right)\\
&+\l ^2 e^{-\l u}\left(\frac{1}{2}Q_u (Du, Du)+fu_t^2 u_{tt}^{-1}\right)+2b\sigma_2(A_u)\\
\geq & -3 \l e^{-\l u} f+\l e^{-\l u}\sigma_1(T_1)\left(c_1+\frac{1}{2}|\nabla u|^2\right)+\l^2 e^{-\l u}(\sigma_2(A_u)+f u_t^2u_{tt}^{-1})\\
&+\left(\frac{\l ^2}{2}-\l\right) e^{-\l u} \langle  T_1(E_u), \nabla u\otimes \nabla u\rangle.
\end{split}
\end{equation}
This completes the proof if $\l$ is sufficiently large. 
\end{proof}

\begin{lemma}There exists a uniform constant $C_3=C_3(\sup f, \sup |\nabla f^{\frac{1}{3}}|, g, |u_0|_{C^1}, |u_1|_{C^1})$ such that
\[
|\nabla u|\leq C_3. 
\]
\end{lemma}

\begin{proof}We take the barrier function
\[
v=|\nabla u|^2+e^{-\l u}+bt^2,
\]
where $\l, b$ are the constants in Lemma \ref{keylemma}. 
We compute
\[
\cL_F(v)=\cL_F(|\nabla u|^2)+\cL_F(e^{-\l u}+bt^2).\]
We have, by Proposition \ref{C1}, that
\[
\cL_F(|\nabla u|^2)\geq 2\nabla f\nabla u-C_1 \sigma_1(T_1)-C_1\sigma_1(T_1)|\nabla u|^2
\]
Hence by Lemma \ref{keylemma}, we have
\[
\cL_F(v)\geq 2\nabla f\nabla u-C_3 f+2\sigma_1(T_1) |\nabla u|^2+f u_t^2 u_{tt}^{-1}.
\]
If $v$ achieves its maximum on the boundary, then we are already done. Otherwise, suppose  $v$ achieves its maximum at $p=(t, x)\in (0, 1)\times M$. Then $\cL_F(v)\leq 0$ at $p$.
Hence it follows that (at $p$)
\[
2\sigma_1(T_1) |\nabla u|^2+f u_t^2 u_{tt}^{-1}\leq 2|\nabla f||\nabla u|+C_3 f
\]
We compute
\[
\sigma_1(T_1)|\nabla u|^2+\sigma_1(T_1) |\nabla u|^2+f u_t^2 u_{tt}^{-1}\geq 3\left(\sigma_1(T_1)^2 u_{tt}^{-1} f u_{t}^2|\nabla u|^4\right)^{\frac{1}{3}}. 
\]
Since $\sigma_1(T_1)^2u_{tt}^{-1}=(n-1)^2 \sigma_1(E_u)^2 u_{tt}^{-1}\geq 2(n-1)^2 \sigma_2(E_u) u_{tt}^{-1}=2(n-1)^2 f$, it follows that (at $p$), 
\[
f^{\frac{2}{3}} |\nabla u|^{\frac{4}{3}}\leq |\nabla f| |\nabla u|+C_3 f. 
\]
This gives the upper bound of $|\nabla u|$ at $p$, and hence the upper bound of $v$. It is not hard to check the dependence of the constants. 
\end{proof}

\begin{rmk}Lemma \ref{keylemma} is essentially proved in \cite{GS2} (for general $k$) and it serves the key to achieve the estimate of $|\nabla u|$. The estimate of $|\nabla u|$ is done in \cite{GS2} (for general $k$) and our argument is a minor modification for $k=2$. 
\end{rmk}

\subsection{$C^2$ estimates}Now we derive the estimates of second order.  Note that $A_u\in \Gamma^+_2$ implies that $\sigma_1(A_u)>0$.  Given the uniform bound on $|\nabla u|$, 
\[
\sigma_1(A_u)=\text{Tr}(A)+\Delta u+\left(1-\frac{n}{2}\right) |\nabla u|^2>0.
\]
This leads to a lower bound of $\Delta u$: there exists a constant $C_2$ such that  $\Delta u+C_2\geq 1. $ 
Moreover, this gives the equivalence of $\sigma_1(A_u)$ and $\Delta u$ in the sense 
\begin{equation}
|\sigma_1(A_u)-\Delta u|\leq C_2. 
\end{equation}
We want to derive  upper bound on $u_{tt}$ and $\Delta u+C_2$ (equivalently, the upper bound of $\sigma_1(A_u)$), which will imply the full hessian bound of $u$ since $A_u\in \Gamma^+_2$, and 
\[
|A_u|^2=\sigma_1(A_u)^2-2\sigma_2(A_u)\leq \sigma_1(A_u)^2.
\] 
The bound on $|\nabla u_t|$ will follow from Proposition \ref{gamma}, in the sense that
\[
u_{tt}\sigma_1(A_u)-|\nabla u_t|^2>0. 
\]
The estimates of second order contain the boundary estimates and the interior estimates. 
The boundary is given by two time slices $\{t=0\}\times M$ and $\{t=1\}\times M$.   The tangential-tangential direction, namely $|\nabla^2u|$ is immediate by the boundary data $|\nabla^2 u_0|, |\nabla^2 u_1|$. While the usual ``harder" part of the normal-normal direction  ($u_{tt}$) follows directly from the equation once the tangential-normal direction ($|\nabla u_t|$) is bounded,
\[
u_{tt}\sigma_2(A_u)=\langle T_1(A_u), \nabla u_t\otimes \nabla u_t\rangle+f.
\]
Note that $\sigma_2(A_u)\geq \delta>0$ at $t=0$ and $t=1$, for some uniform constant $\delta$ depending only on $u_0, u_1$. 
Hence one only needs to bound $|\nabla u_t|$ on the boundary. Such a uniform estimate has been obtained by Gursky-Streets in \cite{GS2} for the equation for all $1\leq k\leq n$, 
\begin{equation}\label{gs}
u_{tt}^{1-k}\sigma_k(E_u)=f
\end{equation}
They stated their results for $E_u^\epsilon=(1+\epsilon)u_{tt}A_u-\nabla u_t\otimes \nabla u_t$ but $\epsilon$ does not play any role in their argument. 
We summarize their results as follows,
\begin{thm}[Gursky-Streets \cite{GS2}]\label{TGS}If $E_u\in \Gamma^+_2$ and $u$ solves \eqref{gs}. There exists a uniform constant $C_3$, such that
\[
\max_{M\times \{0, 1\}}(u_{tt}+|\nabla^2 u|+|\nabla u_t|)\leq C_3.
\]
\end{thm}

Gursky-Streets also obtained interior $C^2$ estimates for \eqref{gs}, depending on the parameter $\epsilon^{-1}$. 
The original computations of $\cL_F(u_{tt})$ and $\cL_F(\Delta u)$ in \cite{GS2} are really involved and impressive. 
Here we offer a variant of such computations and this provides great simplifications. Our treatment should be very standard in nonlinear elliptic theory for concave (convex) operators, in particular over domains of Euclidean spaces.
However, the nonlinear terms of first order in $A_u$ and the curvature of the background metric will bring extra challenge, not only making the computations much more complicated, but also introducing several nonlinear terms which need extra care. 
That is the main difficulty that we overcome to obtain a uniform interior $C^2$ estimates. 

We need some preparations. 
Given a symmetric matrix $R=(r_{ij})$ of $(n+1)\times (n+1)$, we use $r=(r_{ij})$ for the $n\times n$ portion with $ij\neq 0$ and $Y=(r_{01}, \cdots, r_{0n})$.
We write 
\[
F(R)=r_{00}\sigma_2(r)-\langle T_1(r), Y\otimes Y\rangle, \;\text{and}\; G(R)=\log F(R).
\]
We use the standard notation
\[
G^{ij}=\frac{\p G}{\p r_{ij}}=F^{-1} F^{ij}, G^{ij, kl}=\frac{\p^2 G}{\p r_{ij}\p r_{kl}}.
\]
Take the matrix $R$ of the form
\[
R=\begin{pmatrix}u_{tt} & \nabla u_t\\
\nabla u_t& A_u
\end{pmatrix}
\]
Then we write the equation $F(R)=f$ and its equivalent form $G(R)=\log f. $
With this notation, we also record the linearization of $F(R)$. Given a smooth function $\phi$, we have
\begin{equation}\label{second0}
\cL_F(\phi)=F^{ij}\Phi_{ij}, \text{with}\; \Phi= \begin{pmatrix}\phi_{tt} & \nabla \phi_t\\
\nabla \phi_t& \cL_{A_u}\phi
\end{pmatrix}
\end{equation}
We record the derivatives of $F$. 
\begin{prop}\label{F-derivative}We have
\[
G^{ij}=F^{-1} F^{ij}, G^{ij, kl}=F^{-1}F^{ij, kl}-F^{-2}F^{ij}F^{kl}
\]
We compute, for $ij\neq 0$,
\begin{equation}\label{F-derivative2}
\begin{split}
&F^{00}=\sigma_2(r), F^{00, 00}=0,\; F^{00, i0}=0, \;\;F^{00, ij}=T_1(r)^{ij}\\
&F^{i0}=-\langle T_1(r), Y\boxtimes e_i\rangle=F^{0i},\; F^{i0, j0}=-2T_1(r)^{ij}, \;\;F^{i0, kl}=-\langle T_1(e_{kl}), Y\boxtimes e_i\rangle\\
&F^{kl}=\langle T_1(r_{00}r-Y\otimes Y), e_{kl}\rangle, \;\;F^{ij, kl}=r_{00} \langle T_1(e_{ij}), e_{kl}\rangle
\end{split}
\end{equation}
\end{prop}
\begin{proof}This is a straightforward computation. 
\end{proof}

Now we are ready to compute $\cL_F(u_{tt})$ and $\cL_F(\Delta u)$.
\begin{prop}\label{tt} We have the following,
\begin{equation}
\cL_F(u_{tt})=f_{tt}-f_t^2 f^{-1}-f G^{ij, kl} \p_t r_{ij} \p_t r_{kl}-\langle T_1(E_u), 2\nabla u_t\otimes \nabla u_t-|\nabla u_t|^2 g\rangle.
\end{equation}
\end{prop}
\begin{proof}
We compute
\[
\p_t G=G^{ij} \p_t r_{ij}=f_t f^{-1}, \p_t^2 (G)=G^{ij, kl} \p_t r_{ij} \p_t r_{kl}+G^{ij}\p^2_t r_{ij}= f_{tt} f^{-1}- (f_tf^{-1})^2. 
\]
That is

\begin{equation}\label{second1}
G^{ij, kl} \p_t r_{ij} \p_t r_{kl}+F^{-1}F^{ij}\p^2_t r_{ij}=f_{tt} f^{-1}- (f_tf^{-1})^2.
\end{equation}
Now we consider 
\[
(\p^2_t r_{ij})=\p^2_t R=\begin{pmatrix}\p^2_tu_{tt} & \p^2_t\nabla u_t\\
\p^2_t\nabla u_t& \p^2_tA_u
\end{pmatrix}
\]
The main point is that $A_u$, hence $R$ is not linear on $D^2u$.  We compute
\[
\begin{split}
\p^2_t A_u=& \nabla^2 u_{tt}+\nabla u_{tt}\boxtimes \nabla u-(\nabla u_{tt}, \nabla u) g+2\nabla u_t\otimes \nabla u_t-|\nabla u_t|^2 g\\
=& \cL_{A_u} u_{tt}+2\nabla u_t\otimes \nabla u_t-|\nabla u_t|^2 g.
\end{split}\]
Denote $\cR=2\nabla u_t\otimes \nabla u_t-|\nabla u_t|^2 g$ and this is the term coming from the nonlinearity of $A_u$. 
Hence we can write, with $\phi=u_{tt}$,
\[
\p^2_t R=\begin{pmatrix} {\phi}_{tt} & \nabla \phi_t\\
\nabla \phi_t& \cL_{A_u}(\phi)+\cR
\end{pmatrix}
\]
By \eqref{second0} and \eqref{second1}, we get that
\[
G^{ij, kl} \p_t r_{ij} \p_t r_{kl}+F^{-1}\cL_F(u_{tt})+F^{-1}F^{ij}\cR_{ij}=f_{tt} f^{-1}-(f_tf^{-1})^2,
\]
where we use the notation $\cR_{i0}=0$, for $i=0, 1, \cdots, n$. We claim that
\[
F^{ij}\cR_{ij}=T_1(E_u)^{ij}\cR_{ij}=\langle T_1(E_u), \cR\rangle.
\]
But this is straightforward since $F=u_{tt}^{-1}\sigma_2(E_u)$, 
\[
F^{ij}=\langle T_1(E_u), e_{ij}\rangle, ij\neq 0
\] This completes the proof. 
\end{proof}

Next we compute $\cL_F(\Delta u)$ in a similar way. The computations are more involved since not only the nonlinearity of $A_u$, but the background geometry will play an important role. 
\begin{prop}\label{delta1}We have the following,
\begin{equation}\label{delta}
\cL_F(\Delta u)=-fG^{ij, kl} \nabla_p r_{ij}\nabla_p r_{kl}+\Delta f-|\nabla f|^2 f^{-1}-F^{ij}\cR_{1, ij},
\end{equation}
where $\cR_1$ is given in \eqref{r1} and \eqref{r2}. We have the following, 
\[
F^{ij}\cR_{1, ij}=-2u_{tt}^{-1}\langle T_1(E_u), Ric(\nabla u_t, \cdot)\boxtimes \nabla u_t\rangle+\langle T_1(E_u), \cS \rangle.
\]
For simplicity of notation, we identify $Ric(\nabla u_t, \cdot)$ with its dual vector.
We can write $\cS$ as
\[\cS=2\sum_p \nabla\nabla_p u\otimes \nabla \nabla_p u-|\nabla^2 u|^2g+Rm*\nabla^2 u+\cS_0
\]
where $\cS_0$ is a uniformly bounded term (matrix) and $Rm*\nabla^2 u$ denotes two terms of contraction of curvature with $\nabla^2u$ (which we do not need precise expression). 
\end{prop}
\begin{proof}
We compute
\begin{equation}\label{r3}
\Delta G(R)=G^{ij, kl} \nabla_p r_{ij}\nabla_p r_{kl}+F^{-1} F^{ij} \Delta r_{ij}=\Delta f f^{-1}-|\nabla f|^2 f^{-2}. 
\end{equation}
Now we compute
\[
\left(\Delta r_{ij}\right)=\Delta R= \begin{pmatrix}\Delta u_{tt} & \Delta \nabla u_t\\
\Delta \nabla u_t& \Delta A_u
\end{pmatrix}
\]
Recall $A_u=A+\nabla^2 u+\nabla u\otimes \nabla u-|\nabla u|^2 g/2$ and now we compute $\Delta A_u$. 
We need several Bochner-Weitzenbock formula as follows,
\[
\begin{split}
&\Delta \nabla u_t=\nabla \Delta u_t+Ric(\nabla u_t, \cdot), \Delta \nabla^2 u=\nabla^2 \Delta u+Rm*\nabla^2 u+\nabla Rm *\nabla u.\\
&\Delta (\nabla u\otimes \nabla u)=\nabla \Delta u\boxtimes \nabla u+Ric(\nabla u, \cdot)\boxtimes \nabla u+2\nabla \nabla_p u\otimes \nabla \nabla_p u\\
&\Delta \left(\frac{1}{2}|\nabla u|^2\right)=|\nabla^2 u|^2+Ric(\nabla u, \nabla u)+\langle \nabla \Delta u, \nabla u\rangle.
\end{split}
\]
We use $Rm*\nabla^2 u+\nabla Rm *\nabla u$ to denote contraction of terms which we do not need precise expression. 
We can then compute
\[
\Delta R= \begin{pmatrix}(\Delta u)_{tt} & \nabla \Delta u_t+Ric(\nabla u_t, \cdot)\\
\nabla \Delta u_t+Ric(\nabla u_t, \cdot) & \cL_{A_u}(\Delta u)+\cS
\end{pmatrix}
\]
where $\cS$ is the remaining matrix of the form
\begin{equation}\label{r2}
\begin{split}
\cS=&Ric(\nabla u, \cdot)\boxtimes \nabla u+2\nabla \nabla_p u\otimes \nabla \nabla_p u-(|\nabla^2 u|^2+Ric(\nabla u, \nabla u))g\\
&+ \Delta A+Rm*\nabla^2 u+\nabla Rm *\nabla u.
\end{split}
\end{equation}
Denote 
\begin{equation}\label{r1}
\cR_1=\begin{pmatrix}0 & Ric(\nabla u_t, \cdot)\\
Ric(\nabla u_t, \cdot) & \cS
\end{pmatrix}
\end{equation}
Then we can write
\[
\Delta R=\cR_1+\begin{pmatrix}(\Delta u)_{tt} & \nabla \Delta u_t\\
\nabla \Delta u_t & \cL_{A_u}(\Delta u)
\end{pmatrix}
\]
It then follows that
\[
F^{ij}\Delta r_{ij}=\cL_F(\Delta u)+F^{ij}\cR_{1, ij}
\]
Together with \eqref{r3} this completes the proof of \eqref{delta}. The computation of $F^{ij}\cR_{1, ij}$ is straightforward, noting that
\[
F^{i0}=-u_{tt}^{-1}\langle T_1(E_u), \nabla u_t\boxtimes e_i\rangle, i\neq 0. 
\]
\end{proof}

We will need the following estimate, which would be used to take care of the terms coming from the first order terms of $A_u$. 
\begin{lemma}\label{positivity}Let $\phi$ be any smooth function. For $n= 4$, 
\begin{equation}\label{g3}
\left\langle T_1(E_u), -\nabla \phi\otimes \nabla \phi+\frac{|\nabla \phi|^2}{2} g\right\rangle\geq 0
 \end{equation}
 For $n\geq 5$, 
 \begin{equation}\label{g31}
\left\langle T_1(E_u), -\nabla \phi\otimes \nabla \phi+\frac{|\nabla \phi|^2}{2} g\right\rangle\geq \frac{2}{5}\sigma_1(E_u)|\nabla\phi|^2.
 \end{equation}
 \end{lemma}
 \begin{proof}We compute
 \begin{equation}
 \left\langle T_1(E_u), -\nabla \phi\otimes \nabla \phi+\frac{|\nabla \phi|^2}{2} g\right\rangle=\frac{1}{2}\sigma_1(T_1) |\nabla \phi|^2-\langle T_1(E_u), \nabla \phi\otimes \nabla \phi\rangle
 \end{equation}
 It is clear that $
\sigma_1(T_1)=(n-1) \sigma_1(E_u).$ Let $E_u$ be diagonalized with eigenvalues $\l_1\geq\cdots\geq \l_n$. Then we have,
\[
\langle T_1(E_u), \nabla \phi\otimes \nabla \phi\rangle=\sum_i (\sigma_1(E_u)-\l_i)|\nabla_i\phi|^2.
\]
Hence 
\[
\begin{split}
\left\langle T_1(E_u), -\nabla \phi\otimes \nabla \phi+\frac{|\nabla \phi|^2}{2} g\right\rangle= & \sum_i\left(\frac{n-1}{2}\sigma_1(E_u)-(\sigma_1(E_u)-\l_i)\right)|\nabla_i\phi|^2\\
\geq &\left(\frac{n-3}{2}\sigma_1(E_u)+\l_n\right)|\nabla \phi|^2\end{split}
\]
We assume that $\l_n<0$ (otherwise done). Since $E_u\in \Gamma^+_2$, we know $\l_1+\cdots+\l_{n-1}>0$ and 
\[\sigma_2(E_u)=\l_n(\l_1+\cdots+\l_{n-1})+\sigma_2(\l_1, \cdots, \l_{n-1})>0.\]
When $n=4$, 
it is then sufficient to show that
$\sigma_1(E_u)+2\l_4>0.$
We have 
\[
\sigma_1(E_u)+2\l_4>(\l_1+\l_2+\l_3)-3\sigma_2(\l_1, \l_2, \l_3)(\l_1+\l_2+\l_3)^{-1}\geq 0. 
\]
This follows since we have
\[
\l_1^2+\l_2^2+\l_3^2\geq \l_1\l_2+\l_2\l_3+\l_3\l_1. 
\]
When $n\geq 5$, we want to show that \[
\frac{n-3}{2}\sigma_1(E_u)+\l_n\geq \frac{2}{5}\sigma_1(E_u)
\]
That is 
\[
\left(\frac{n-3}{2}-\frac{2}{5}\right)(\l_1+\cdots+\l_{n-1})+\left(\frac{n-3}{2}+\frac{3}{5}\right)\l_n\geq 0.
\]
Again this follows from an elementary inequality and 
\[
\l_n>-\sigma_2(\l_1, \cdots, \l_{n-1})(\l_1+\cdots+\l_{n-1})^{-1}. 
\]
 \end{proof}
 
\begin{rmk}When $n=4$, one can actually get a more precise inequality, for $E\in \Gamma^+_2$
\[
\sigma_1(E)+2\l_4> \frac{2\sigma_2(E)}{\sigma_1(E)}
\]
And the best constant is $2$ on the right hand side, with the example of $(1, 1, 1, -1+\epsilon)$.  
\end{rmk}

The interior estimate of $u_{tt}$ now becomes immediate ($n\geq 4$), 
\begin{lemma}For $n\geq 4$, there exists a constant $C_3$ such that
\[
u_{tt}\leq C_3. 
\]
\end{lemma}
\begin{proof}By the concavity of $G$, Lemma \ref{positivity} and Proposition \ref{tt}, we have
\[
\cL_F(u_{tt})\geq f_{tt}-f_t^2 f^{-1}.
\]
It then follows that, using \eqref{E9}, 
\[
\cL_F(u_{tt}+u_t^2)\geq 2u_t f_t+2f u_{tt}+f_{tt}-f_t^2 f^{-1}.
\]
If $u_{tt}+u_t^2$ achieves its maximum on the boundary, then by Theorem \ref{TGS} we are done. Otherwise at the maximum point of $u_{tt}+u_t^2$, we have
\[
2u_t f+2f u_{tt}+f_{tt}-f_t^2 f^{-1}\leq 0
\]
This is sufficient to bound $u_{tt}$ by a uniform constant $C_3$, where $C_3$ depends on the boundary estimate of $u_{tt}$ and $-f_{tt}f^{-1}, |f_t|f^{-1}$ in addition. 
\end{proof}
When $n\geq 5$, we can get the interior bound of $\Delta u$ as follows,
\begin{lemma}When $n\geq 5$, there exists a uniform constant $C_3$ such that
\[
\Delta u\leq C_3. 
\]
\end{lemma}
\begin{proof}
By the concavity of $G$ and Proposition \ref{delta1}, we have
\[
\cL_F(\Delta u)\geq \Delta f-|\nabla f|^2f^{-1}+2u_{tt}^{-1}\langle T_1(E_u), Ric(\nabla u_t, \cdot)\boxtimes \nabla u_t\rangle-\langle T_1(E_u), \cS\rangle. 
\]
 We write $\cS=\cS_0+\cS_1+Rm*\nabla^2 u$, with
 $$\cS_1=2\sum_p \nabla\nabla_p u\otimes \nabla \nabla_p u-|\nabla^2 u|^2g.$$
 By Proposition \eqref{positivity}, \[-\langle T_1(E_u), \cS_1\rangle >\frac{4}{5}\sigma_1(E_u) |\nabla^2 u|^2.\] 
 We also estimate
 \[
 2u_{tt}^{-1}\langle T_1(E_u), Ric(\nabla u_t, \cdot)\boxtimes \nabla u_t\rangle\geq -C_1 u_{tt}^{-1}\sigma_1(T_1) |\nabla u_t|^2>-C_1 \sigma_1(T_1) \sigma_1(A_u).
 \]
 Since $\cS_0$ is a uniformly bounded term, we obtain,
 \[
 \cL_F(\Delta u)\geq  \Delta f-|\nabla f|^2f^{-1}-C_1 \sigma_1(T_1) \sigma_1(A_u)-C_2 |T_1(E_u)| (|\nabla^2 u|+1)+\frac{4}{5}\sigma_1(E_u) |\nabla^2 u|^2. 
 \]
Note that $\sigma_1(T_1)=(n-1)\sigma_1(E_u)>|T_1(E_u)|$, we obtain
\[
\cL_F(\Delta u)\geq \Delta f-|\nabla f|^2 f^{-1}-\tilde C_2 \sigma_1(E_u) (|\nabla^2 u|+1)+\frac{4}{5}\sigma_1(E_u) |\nabla^2 u|^2. 
\]
Suppose $\Delta u$ obtains its maximum at an interior point $p$ (otherwise we are done). 
At the interior maximum of $\Delta u$, we have obtained,
\[
\Delta f-|\nabla f|^2 f^{-1}-\tilde C_2 \sigma_1(E_u) (|\nabla^2 u|+1)+\frac{4}{5}\sigma_1(E_u) |\nabla^2 u|^2\leq 0. 
\]
We can assume $|\nabla ^2 u|\geq 100\tilde C_2+100$ at the maximum of $\Delta u$ (otherwise we are done),  then
\[
\frac{1}{5}\sigma_1(E_u) |\nabla^2 u|^2\leq -\Delta f+|\nabla f|^2 f^{-1}
\]
By Proposition \ref{gamma} (see \eqref{gamma3}), we have (at $p$), 
\[
\sigma_2(A_u)^{-1}\sigma_1(A_u) |\nabla^2 u|^2 \leq -f^{-1}\Delta f+|\nabla f|^2 f^{-2}. 
\]
This is sufficient to get a uniform upper bound of $\Delta u$. 
\end{proof}

The estimates of $\Delta u$ (for $n\geq 5$) is rather straightforward given the strictly lower bound of the quadratic form in Lemma \ref{positivity}. When $n=4$, such a positivity is too weak and 
the interior estimate of $\Delta u$ is rather subtle.

\begin{lemma}When $n=4$, there exists a uniform constant $C_3$ such that
\[
\Delta u\leq C_3. 
\]
\end{lemma}
\begin{proof}First note that  $c_1(T_1)=3\sigma_1(E_u)$, which will be used in the following.
We consider  \[\max_{p\in [0, 1]\times M} |\nabla u_t|^2 u_{tt}^{-1}=K\]
We want to emphasize that we do not have \emph{a priori uniform} bound for $K$.  We construct the barrier function as 
\[
v=\Delta u+\frac{1}{2}K t^2+|\nabla u|^2-\l u+\l t^2. 
\]
The choice of term $K t^2/2$ is essential for us. 
By Proposition \ref{delta} and the concavity of $G$, we have
\[
\cL_F(\Delta u)\geq -C_3 f+2u_{tt}^{-1}\langle T_1(E_u), Ric(\nabla u_t, \cdot)\boxtimes \nabla u_t\rangle -\langle T_1(E_u), \cS\rangle. 
\]
We write $\cS=\cS_0+\cS_1+Rm*\nabla^2 u$ with  $\cS_0$ a bounded term.
Hence
\[
 -\langle T_1(E_u), \cS\geq -C_1\sigma_1(T_1) (|\nabla^2 u|+1)-\langle T_1(E_u), \cS_1\rangle. 
\]
Since when $n=4$, we know $Ric>0$ \cite{GWV} for any metric in $\Gamma^+_2$, hence 
\[
u_{tt}^{-1}\langle T_1(E_u), Ric(\nabla u_t, \cdot)\boxtimes \nabla u_t\rangle>0\]
We should mention that the positivity of Ricci is not essential since we have
\[
\cL_F(\l t^2)=2\l \sigma_2(A_u)
\]  
and we have the control from the equation
\[
\sigma_2(A_u)>\langle T_1(A_u), \nabla u_t\otimes \nabla u_t \rangle =u_{tt}^{-1}\langle T_1(E_u), \nabla u_t\otimes \nabla u_t \rangle.
\]
Hence we have (regardless of positive Ricci), 
\begin{equation}\label{last1-0}
\cL_F(\Delta u+\l t^2)\geq -C_3 f-C_1\sigma_1(E_u) (|\nabla^2 u|+1)-\langle T_1(E_u), \cS_1\rangle,
\end{equation}
where the major bad term is $-C_1 \sigma_1(E_u) |\nabla^2u|$ coming from the curvature of the background metric.
We have \[T_1(A)\geq c_0 g=3c_1g,\] with a uniformly positive lower bound $c_1>0$. Hence we compute, 
\begin{equation}
\begin{split}\label{last1-1}
\cL_F(-\l u)=&-3\l f+\l \langle T_1(E_u), A-\nabla u\otimes \nabla u+\frac{1}{2}|\nabla u|^2 g\rangle\\
\geq &-3\l f+3\l c_1 \sigma_1(E_u)\end{split}
\end{equation}
Next we compute
 \[
 \cL_F\left(\frac{1}{2}Kt^2+|\nabla u|^2\right)\geq K\sigma_2(A_u)+\sum_iQ_u(Du_i, Du_i)+\frac{2f}{u_{tt}}|\nabla u_t|^2-C_3|\nabla f|-C_3\sigma_1(E_u). 
 \]
 We claim there exists a uniformly positive constant $0<\epsilon_0\leq 1/2$ (depending on the lower bound of $u_{tt}^{-1}$) such that
 \begin{equation}\label{last2-1}
 K\sigma_2(A_u)+\sum_iQ_u(Du_i, Du_i)\geq \sum_i \epsilon_0\langle T_1(E_u), \nabla \nabla_i u\otimes \nabla \nabla_i u\rangle
 \end{equation}
 Given the claim at the moment, we observe that
 \begin{equation}\label{last3-1}
 -\langle T_1(E_u), \cS_1\rangle+\sum_i \epsilon_0\langle T_1(E_u), \nabla \nabla_i u\otimes \nabla \nabla_i u\rangle\geq 3\epsilon_0\sigma_1(E_u) |\nabla^2 u|^2. 
 \end{equation}
 Finally we reach at, combining \eqref{last1-0}, \eqref{last1-1}, \eqref{last2-1} and \eqref{last3-1},
 \begin{equation}
 \label{last}
 \cL_F(v)\geq 3\epsilon_0 \sigma_1(E_u) |\nabla^2 u|^2-C_1\sigma_1(E_u)|\nabla^2 u|+(3\l c_1-C_1-C_3)\sigma_1(E_u)-C_3  f. \end{equation}
 If $v=\Delta u+\frac{K}{2} t^2+|\nabla u|^2-\l u+\l t^2$ obtains its maximum on the boundary, then we are done (since $K$ is uniformly bounded on the boundary by Gursky-Streets' boundary estimates). 
 Otherwise $v$ achieves its maximum at an interior point $p$, it follows that (at $p$)
 \begin{equation}\label{last0}
3\epsilon_0 \sigma_1(E_u) |\nabla^2 u|^2-C_1\sigma_1(E_u)|\nabla^2 u|+(3\l c_1-C_1-C_3)\sigma_1(E_u)-C_3 f \leq 0.
 \end{equation}
 We choose $\l$ sufficiently large such that $3\l c_1-C_1-C_3>0$.  We claim that this is sufficient to bound $|\nabla^2 u|$ at $p$, \begin{equation}\label{claim0}
 |\nabla^2 u|(p)\leq C_3. 
 \end{equation}
 We can assume $|\nabla^2 u|(p)$ satisfies  $|\nabla^2 u|(p)\geq C_1 \epsilon_0^{-1}$  (otherwise done) and hence
 \[
 \epsilon_0\sigma_1(E_u) |\nabla^2 u|^2-C_1\sigma_1(E_u)|\nabla^2 u|\geq 0.
 \]
  Then by \eqref{last0}, we have
 \[
\epsilon_0 \sigma_1(E_u) |\nabla^2 u|^2\leq C_3 f.
 \]
 Note that by \eqref{gamma3}, we get that (at $p$)
 \[
 \frac{\sigma_1(A_u)}{\sigma_2(A_u)}|\nabla^2 u|^2\leq C_3 \epsilon_0^{-1}. 
 \]
 This establishes the claim \eqref{claim0}. 
 Clearly we have $\Delta u(p)\leq n |\nabla^2 u|(p)$. Since $v\leq v(p)$, we have obtained 
 \[
 \Delta u\leq v\leq v(p)\leq C_3+\frac{K}{2}. 
 \]
In other words, we have
\[
\sup \Delta u\leq C_3 +\frac{K}{2}. 
\]
 Note that $\Delta u-C_2\leq \sigma_1(A_u)\leq \Delta u+C_2$, we get
 \[
\sup \sigma_1(A_u)\leq C_3+\frac{K}{2}.
 \]
 We observe that
 \[
 \sigma_1(A_u)-\frac{|\nabla u_t|^2}{u_{tt}}=u_{tt}^{-1}\sigma_1(E_u) >0
 \]
 Hence $K<\sup \sigma_1(A_u)$, and we have proved that
 \[
 \sup \sigma_1(A_u)\leq C_3. 
 \]
 This gives the uniformly upper bound of $\Delta u$ given the following proposition. 
\end{proof}
We establish \eqref{last2-1} right now.

\begin{prop}We have the following, 
\[K\sigma_2(A_u)+\sum_iQ_u(Du_i, Du_i)\geq \sum_i \epsilon_0\langle T_1(E_u), \nabla \nabla_i u\otimes \nabla \nabla_i u\rangle\]
\end{prop}
\begin{proof}Recall
\[
Q_u(Du_i, Du_i)=\frac{2}{u_{tt}}\left\langle T_1(E_u), u_{tt}\nabla \nabla_i u\otimes \nabla\nabla_i u-\nabla_i u_t \nabla u_t\boxtimes \nabla \nabla_i u+\frac{|\nabla_i u_t|^2}{u_{tt}} \nabla u_t\otimes \nabla u_t\right\rangle.
\]
And we have
\[
\begin{split}
K\sigma_2(A_u)\geq &\frac{|\nabla u_t|^2}{u_{tt}}  u_{tt}^{-1}\langle T_1(A_u), \nabla u_t\otimes \nabla u_t\rangle\\
= & \frac{1}{u_{tt}^{3}} \langle T_1(E_u), |\nabla u_t|^2 \nabla u_t\otimes \nabla u_t\rangle\\
\geq& \frac{2}{u_{tt}} \langle T_1(E_u), 2\epsilon_0 u_{tt}^{-1}|\nabla u_t|^2 \nabla u_t\otimes \nabla u_t\rangle,
\end{split}
\]
for some uniformly positive constant $\epsilon_0\leq 1$ such that $u_{tt}^{-1}\geq 4\epsilon_0$. It then follows that
\[
\begin{split}
K\sigma_2(A_u)+\sum_i Q_u(Du_i, Du_i)\geq& \frac{2}{u_{tt}}\left\langle T_1(E_u), u_{tt}\nabla \nabla_i u\otimes \nabla\nabla_i u-\nabla_i u_t \nabla u_t\boxtimes \nabla \nabla_i u+(1+\epsilon_0)\frac{|\nabla_i u_t|^2}{u_{tt}} \nabla u_t\otimes \nabla u_t\right\rangle\\
\geq &\epsilon_0 \sum_i\langle T_1(E_u), \nabla \nabla_i u\otimes \nabla \nabla_i u\rangle
\end{split}
\]
This completes the proof. 
\end{proof}

\begin{rmk}
Even though we have the positivity of the following,
\[
-\langle T_1(E_u), \cS_1\rangle\geq 2\frac{\sigma_2(E_u)}{\sigma_1(E_u)} |\nabla ^2 u|^2,
\]
this good term solely is not sufficient. 
Compared with the bad term $-\sigma_1(E_u) |\nabla^2 u|$, it is not hard to see that 
\[
\frac{\text{bad term}}{\text{good term}}``="\frac{\sigma_1(A_u)-u_{tt}^{-1} |\nabla u_t|^2}{\sigma_2(A_u)}
\]
If we ignore the term $u_{tt}^{-1}|\nabla u_t|^2$, there is no way to control this ratio directly in view of an example $\text{diag}(\epsilon^{-1}, \epsilon, 0, 0)$. 
The positive Ricci curvature does not play an essential role, since
\[
\langle T_1(E_u), Ric(\nabla u_t, \cdot)\boxtimes \nabla u_t\rangle``="\sigma_2(A_u). 
\]
This term is harmless either way (with positive or negative sign) and it is also helpless, by the same reason. 
The quadratic term coming from $\cL_F(|\nabla u|^2)$ reads
\[
\sum_iQ_u(Du_i, Du_i)+\frac{2f}{u_{tt}}|\nabla u_t|^2
\]
We can also argue that
\[
\sum_iQ_u(Du_i, Du_i)+\frac{2f}{u_{tt}}|\nabla u_t|^2\geq 2f \frac{\langle T_1(E_u), \nabla\nabla_i u\otimes \nabla \nabla_i u\rangle}{\langle T_1(E_u), A_u\rangle}
\]
But this good term is not sufficient to dominate the bad term $-\sigma_1(E_u)|\nabla^2 u|$ in general. The essential inequality for us  is
\[
\sum_iQ_u(Du_i, Du_i)+K\sigma_2(A_u)-\langle T_1(E_u), \cS_1\rangle \geq 3 \epsilon_0 \sigma_1(E_u) |\nabla^2 u|^2. 
\]

\end{rmk}

\subsection{Solve the equation and uniqueness}
In this section we prove Theorem \ref{mainthm1} and Theorem \ref{mainthm2}. With the estimates we derived above, the proof is standard and we keep it brief.
\begin{proof}[Proof of Theorem \ref{mainthm1} and Theorem \ref{mainthm2}]
First we prove the uniqueness when $f>0$. The argument is a standard comparison, using the ellipticity and the concavity (compare $C^0$ estimates). Suppose $\tilde u$ and $u$ both solve the equation 
\[
u_{tt}\sigma_2(A_u)-\langle T_1(A_u), \nabla u_t\otimes \nabla u_t\rangle =f
\]
with the same boundary data. We want to prove that $\tilde u=u$. Suppose otherwise, we can assume at some interior point, $\tilde u>u$. Hence for some small $a>0$, we have at some interior point,
\[
\tilde u+at(t-1)>u
\]
Denote $v=\tilde u+at(t-1)$.
Consider the maximum point $p$ of $v-u$, we have
\[
D(v-u)=0, D^2 (v-u)\leq 0.
\]
On one hand, 
\[
F(v_{tt}, A_v, \nabla v_t)=(\tilde u_{tt}+2a)\sigma_2(A_{\tilde u})-\langle T_1(A_{\tilde u}), \nabla \tilde u_t\otimes \nabla \tilde u_t\rangle>f
\]
On the other hand, we have at $p$ (since $D^2(u-v)\geq 0, D(u-v)=0$)
\[
\log F(u_{tt}, A_u, \nabla u_t)-\log F(v_{tt}, A_v, \nabla v_t)\geq F^{-1} \cL_F (u-v)\geq 0,
\]
where $\cL_{F}$ is the linearized operator of $F$ at $v$; the non-negativity of $\cL_F(u-v)$ at $p$ follows the same argument in Proposition \ref{c0}. 
This is a contradiction. This proves $\tilde u\leq u$. Interchanging $\tilde u$ and $u$ we get $u\leq \tilde u$. Hence we have $\tilde u=u$. This proves the uniqueness.

Given $u_0, u_1$ two \emph{admissible} boundary datum, consider $w=(1-t)u_0+t u_1+at(t-1)$ for sufficiently large $a$. 
We write \[F(w_{tt}, A_w, \nabla v_w)=w_{tt}\sigma_2(A_w)-\langle T_1(A_w), \nabla w_t\otimes \nabla w_t\rangle =f_0.\]
When $a$ is sufficiently large, $f_0>0$ and hence $E_w=w_{tt}A_w-\nabla w_t\otimes \nabla w_t\in \Gamma^+_2$.
We use the continuity method to solve the equation. Let $f_s=sf+(1-s) f_0$. We want to solve the equation for $u^s$, $s\in [0, 1]$, 
\begin{equation}\label{s1}
F(u^s_{tt}, A_{u^s}, \nabla u^s_t\otimes \nabla u^s_t)=f_s
\end{equation}
We choose the normalization condition for the boundary datum $u_0, u_1$ by adding $c_0 t+c_1$ for some constants $c_0, c_1$, such that \eqref{normalization} holds. Note that the change of boundary datum does not change $D^2 w$ (hence does not change $f_0$). When $s=0$, $u^s=w$ solves the equation. The linearized operator $\cL_{F_s}$ is invertible for $s\in [0, 1]$ for $u^s\in \Gamma^+_2$ (see the proof of uniqueness) and hence there exists a unique solution of the linearized equation
\[
\cL_{F_s}(h)=\tilde f
\]
with the zero boundary datum $h(\cdot, 0)=h(\cdot, 1)=0$ for any smooth $\tilde f$. It follows that
the set
$S:=\{s\in [0, 1]: \eqref{s1}\; \text{has a unique solution}\}$ is open in $[0, 1]$. We need to prove the set $S$ is also closed. Suppose $s_i\in [0, 1)$ such that $u^{s_i}$ is the unique solution of \eqref{s1}. Suppose $s_i\rightarrow s_0$. By our a priori estimates, $u^{s_i}$ have uniformly bounded $C^{1, 1}$ norm. Given the concavity of $G$, Evans-Krylov's theory applies and we get uniform $C^{2, \alpha}$ estimates of $u^{s_i}$ for some $\alpha$. The boot-strapping argument then implies the uniform bound $C^{l, \alpha}$  for any $l\geq 2$. Hence by passing to subsequence if necessary, $u^{s_i}$ converges to $u^{s_0}$ smoothly. By convergence we see that $u^{s_0}$ solves the equation \eqref{s1} for $s=s_0$. This proves the existence \eqref{s1} for $s=s_0$, and hence proves the closeness of $S$.  To prove Theorem \ref{mainthm2}, we fix $f>0$ as above and consider the equation
\[
F(u^s_{tt}, A_{u^s}, \nabla u^s_t\otimes \nabla u^s_t)=sf.
\]
By the a prior estimates established above, $u^s$ has uniformly bounded $C^{1, 1}$-norm. Indeed we know $u^s$ is decreasing in $s$ by a comparison principle below.
When $s\rightarrow 0$, $u^s$ converges strongly in $C^{1, \alpha}$ to $u^0$ such that $u^0\in C^{1, 1}$ (such a limit $u^0$ is indeed unique). In particular $u^0$ solves the equation almost everywhere (the strong solution) with the uniform $C^{1, 1}$ bound. 
\end{proof}

We state a comparison principle below, which can be proved similarly as the proof of uniqueness. 

\begin{lemma}Let $u, v$ be two smooth functions on $[0, 1]\times M$. 
Suppose $A_u\in \Gamma^+_2$ and $A_v\in \Gamma^+_2$. 
If 
\[
F(u_{tt}, A_u, \nabla u_t)\geq F(v_{tt}, A_v, \nabla v_t)>0
\]
and $u, v$ have the same boundary datum, then $u\leq v$. 
Moreover, if we have
\[
F(u_{tt}, A_u, \nabla u_t)>F(v_{tt}, A_v, \nabla v_t)>0.
\]
Then $v>u$ for points in $(0, 1)\times M$ (suppose $u, v$ have the same boundary datum).
In general if
\[
F(u_{tt}, A_u, \nabla u_t)= F(v_{tt}, A_v, \nabla v_t)>0
\]
then we have
\[
\max_{[0, 1]\times M} |u-v|=\max_{\{0, 1\}\times M}|u-v|
\]
\end{lemma}

The comparison principle above relies on the fact that $f>0$ and $A_u, A_v$ are in interior of the cone $\Gamma^+_2$. We can also have a version of comparison principle if one function  is on the boundary and satisfies the homogeneous equation. More precisely, we have

\begin{lemma}
Suppose $u\in C^2$ is admissible such that
\[
F(u_{tt}, A_u, \nabla u_t)>0.
\]
Suppose a $C^2$ function $v\in [0, 1]\times M \rightarrow \R$ satisfies 
\[
v_{tt}\geq 0, A_v\in \bar \Gamma^+_2, F(v_{tt}, A_v, \nabla v_t)=0.
\]
If $u=v$ on $M\times \{0, 1\}$, then $v>u$ for any point in $(0, 1)\times M$. 
\end{lemma}
\begin{proof}We argue by contradiction. Suppose $v\leq u$.  Then $u-v$ obtains its maximum at an interior point $p$. At $p$, we have
\[
D^2(u-v)\leq 0, D(u-v)=0.
\]
In particular, at $p$,
\[
u_{tt}\leq v_{tt},\; \text{and}\;\; A_u\leq A_v
\]
It follows that $A_v\in \Gamma^+_2$ (at $p$) since
\[
\sigma_2(A_v)-\sigma_2(A_u)=\int_0^1 \langle T_1(sA_v+(1-s)A_u), A_v-A_u\rangle ds\geq 0.
\]
Choose $b>0$ sufficiently small such that at $p$, 
\[
2b \sigma_2(A_v)(p)<F(u_{tt}, A_u, \nabla u_t)(p)
\]
Take $w=v+bt(t-1)$. Then at $p$, we have (by concavity of $\log F$),
\[
\log F(w_{tt}, A_w, \nabla w_t)(p)-\log F(u_{tt}, A_u, \nabla u_t)(p)\geq F^{-1} \cL_F (w-u)(p)
\]
where $F^{-1}\cL_F$ takes value of $u$ at $p$. However, $D^2(w-u)\geq 0$ and $\nabla w=\nabla v$ at $p$. This follows that $\cL_F(w-u)(p)\geq 0$.
This contradicts the choice of $b$. 
\end{proof}

\begin{rmk}We conjecture that the solution $u^0$ constructed is the unique solution of the geodesic equation with fixed boundary datum. However, the comparison principle we derived is not strong enough to prove uniqueness. From now on we choose $f=1$ and consider the equation, for $s\in (0, 1]$,
\[
F(u_{tt}, A_u, \nabla u_t)=s.
\]
We refer this construction $u^s$ as ``the approximating geodesic" and the limit $u=u^0=\lim_{s\rightarrow 0} u^s$ as ``the geodesic", even though we do not prove the uniqueness of the geodesic equation. By comparison principle we can see that $u^0$ is canonical, in the sense that for any smooth $f>0$, the solutions 
\[
F(u_{tt}, A_u, \nabla u_t)=sf
\]
will have the same limit when $s\rightarrow 0$. However this simply asserts the uniqueness of the limit solution regardless of the choice of approximating process, but is not sufficient for the uniqueness of the geodesic equation itself (there might be a solution which is not constructed through the approximating process). On the other hand, the uniqueness does not play an important role for geometric applications. 
\end{rmk}

\section{Appendix}
\subsection{The Donaldson operator and the Gursky-Streets operator}
Denote the matrix $R=(r_{ij})$ for $i, j\in \{0, \cdots, n\}$ and $r=(r_{ij})$ for $ij\neq 0$, $Y=(r_{01}, \cdots, r_{0n})$.
Donaldson has introduced an operator about a decade ago \cite{Donaldson}, 
\[
Q(R)= r_{00}\sigma_1(r)-|Y|^2,
\]
We can write this operator as ($T_0(r)=I$),
\begin{equation}\label{donaldson}
Q(R)=r_{00} \sigma_1(r)-\langle T_0(r), Y\otimes Y\rangle.
\end{equation}
Hence the Donaldson operator is a first operator ($k=0$) in the following family (for $k\leq n$).
\begin{equation}\label{gurskys1}
F_k(R)=r_{00}\sigma_k(r)-\langle T_{k-1}(r), Y\otimes Y\rangle,
\end{equation}
which we call the Gursky-Streets operator. One requires a positivity condition that $r\in \Gamma^+_k$ and $F_k(R)>0$. We introduce $\tilde r=r_{00} r-Y\otimes Y$, Gursky-Streets have the following observation,
\[
r_{00}^{1-k}\sigma_k(\tilde r)=F_k(R)
\]
These operators are not symmetric for $R$ (only symmetric for $r$). Rather it has one special direction (corresponding to $r_{00}$). For a smooth function $u\in \R\times M\rightarrow \R$, take $R$ of the form
\[
R(D^2 u)= \begin{pmatrix}
u_{tt} & \nabla u_t\\
\nabla u_t & \nabla^2 u+S_u,
\end{pmatrix}
\]
where $S_u$ denotes a lower order term, then $F_k(R)=f$ defines  a family of second order differential equations. These equations are constructed as  geodesic equations of interesting infinitely dimensional Riemannian structure, coming from K\"ahler geometry and conformal geometry for example \cite{M2, Donaldson, GS2}. 
In Donaldson's setting, one can take $S_u=g/n$. We should mention that $S_u$ can be taken as any positive definite two tensors one can easily see that there is no any essential difference.
In Gursky-Streets' setting, $A_u$ is the Schouten tensor of a conformal metric $e^{-2u}g$ and
\[
\nabla^2 u +S_u=A_u=\nabla^2 u+A+\nabla u\otimes \nabla u-\frac{1}{2}|\nabla u|^2 g. 
\]

We should emphasize that we can only prove our results ($C^2$ estimates) for $n\geq 4$. The essential reason is that the first order terms $\nabla u\otimes \nabla u-\frac{1}{2}|\nabla u|^2 g$ are nonlinear 
and it would lead to a quadratic form of the type
\[
q=\left\langle T_1(E_u), -\nabla \phi\otimes \nabla \phi+\frac{1}{2}|\nabla \phi|^2 g\right\rangle. 
\]
When $n=3$, this quadratic form can have negative eigenvalues (for $E_u\in \Gamma^+_2$ and we denote its eigenvalues as $\l_1\geq \l_2\geq \l_3$). This would lead to a negative term of the form $(\l_1+\l_2+3\l_3)|\nabla^2_3 u|^2$ in the estimate of $\Delta u$ (similar situation also happens when one considers estimate of $u_{tt}$). 
For example, if $E_u$ has eigenvalues of the form $\l (1, 1, -1+\epsilon)$, then $\l_1+\l_2+3\l_3=-\l (2-3\epsilon)$ is negative and in the same order of $\sigma_1(E_u)$ when $\epsilon$ is sufficiently small. 
 Even one takes $\log (\Delta u+C)$ (leaving aside the additional difficulties by taking logarithm), this term is bad as the order of
$-\sigma_1(E_u)|\nabla^2_3 u|^2$, which exceeds the order of all the good terms (such as $\cL_F(e^{-u}), \cL_F(t^2)$) from the ``subharmonic" functions. 
When $n\geq 5$, the argument is straightforward since the strict positivity of the quadratic form $q$. When $n=4$, we need to explore the positivity of $q$ in a subtle way.

\subsection{The concavity of $G$}
Donaldson \cite{Donaldson} proved that his operator satisfies the following concavity using some elementary Lorentz geometry. Given $R_1, R_2$ satisfying the assumption ($r\in \Gamma^+_1$ and $F_1(R)>0$) such that $Q(R_1)=Q(R_2)$, 
then 
\[
Q((1-s)R_1+sR_2)\geq Q(R_1).
\]
This proves that $Q$ (instead of $\log Q$) is concave on its \emph{level set}. It is straightforward to see that it is equivalent to the fact that $\log Q$ is concave. The concavity of $\log Q$ plays an important role to solve Donaldson's equation. One can also argue the concavity of $\log Q$ by an elementary inequality \cite{Chen-He}.

\begin{lemma}The function
\[
\log (xy-\sum z_i^2)
\]
is concave for $x>0, xy-\sum z_i^2>0$
\end{lemma}

\begin{proof}The function is obviously smooth and we need to argue, 
\[
2\log \left(\frac{x+\tilde x}{2}\cdot \frac{y+\tilde y}{2}-\sum \left(\frac{z_i+\tilde z_i}{2}\right)^2\right)\geq \log (xy-\sum z_i^2)+\log (\tilde x\tilde y-\sum \tilde z_i^2)
\]
Denote, for $\l, a>0$, 
\[
xy-\sum z_i^2=a, \tilde x \tilde y-\sum \tilde z_i^2=\l^2 a.
\]
We need to show that
\[
\frac{x+\tilde x}{2}\cdot \frac{y+\tilde y}{2}-\sum \left(\frac{z_i+\tilde z_i}{2}\right)^2\geq \l a 
\]
Write \[x=\frac{1}{y} (a+\sum z_i^2)\; ; \tilde x=\frac{1}{\tilde y}(\sum z_i^2+\l^2 a)\]
This results in proving the following,
\[
\left(\frac{1}{y} (a+\sum z_i^2)+\frac{1}{\tilde y}(\sum z_i^2+\l^2 a)\right)\left(y+\tilde y\right)-\sum \left(z_i+\tilde z_i\right)^2\geq 4\l a. 
\]
This is elementary.
\end{proof}
\begin{rmk}A quick way to see the concavity is to write the function $xy-\sum z_i^2=u^2-v^2-\sum z_i^2$ with $u=(x+y)/2, v=(x-y)/2$. Then this is a standard example in Garding's theory of hyperbolic polynomials. Hence one can actually get that $(xy-\sum z_i^2)^{1/2}$ is concave. 
\end{rmk}

Now we establish Lemma \ref{concavity}, the concavity of $G=\log F$. 
 \begin{thm} Given $r\in \Gamma^+_2$ and $F=F_2(R)>0$, then $\log F$ is concave. 
 \end{thm} 
\begin{proof}We need to show that
\[
\log \left(r_{00}\sigma_2(r)-\langle T_1 (r), Y\otimes Y \rangle \right)
\]
is concave for $r\in\Gamma^+_2$ and $r_{00}\sigma_2(r)-\langle T_1 (r), Y\otimes Y \rangle >0.$ In other words, we want to show that, for any $s\in [0, 1]$,
\[
\log F((1-s)R+s\tilde R)\geq (1-s) \log F(R)+s\log F(\tilde R)
\]
Since $\log F$ is smooth on $R$, we only need to prove for $s=1/2$. Denote \[x=r_{00}, \tilde x=\tilde r_{00}, \bar x=\frac{x+\tilde x}{2}\]
We also use $T=T_1(r)$ and $\tilde T=T_1(\tilde r)$
We also use the notation $\bar r, \bar Y, \bar T$ to denote the average for simplicity.  
We need to show
\begin{equation}\label{c10}
2\log \left(\bar x\sigma_2(\bar r)-\bar T(\bar Y, \bar Y)\right)\geq \log \left(x\sigma_2(r)-T(Y, Y)\right)+\log\left(\tilde x\sigma_2(\tilde r)-\tilde T(\tilde Y, \tilde Y)\right)
\end{equation}
Denote, for $\l, a>0$,
\[
x\sigma_2(r)-T(Y, Y)=a\;; \tilde x \sigma_2(\tilde r)-\tilde T(\tilde Y, \tilde Y)=\l^2 a. 
\]
We can write 
\[
x=\sigma_2(r)^{-1}(a+T(Y, Y)), \tilde x=\sigma_2(\tilde r)^{-1}(\l^2a+\tilde T(\tilde Y, \tilde Y))
\]
For simplicity we also use the notations $T=T(Y, Y), \tilde T=\tilde T(\tilde Y, \tilde Y), \sigma_2=\sigma_2(r), \tilde \sigma_2=\sigma_2(\tilde r)$ etc when there is no confusion.
We need to show 
\begin{equation}\label{c11}
\left(\frac{\bar \sigma_2 (a+T)}{2\sigma_2}+\frac{\bar \sigma_2(\l^2a+\tilde T)}{2\tilde \sigma_2}-\bar T\right)\geq \l a. 
\end{equation}
By the concavity of $\sqrt{\sigma_2}$ (or rather the concavity of $\log \sigma_2$), we have
\[
\frac{\bar \sigma_2}{2\sigma_2}+\l^2 \frac{\bar \sigma_2}{2\tilde \sigma_2}\geq \l . 
\]
By \eqref{c10} and \eqref{c11}, this reduces to show
\begin{equation}
\frac{\bar \sigma_2 T}{2\sigma_2}+\frac{\bar \sigma_2\tilde T}{2\tilde \sigma_2}-\bar T\geq 0.
\end{equation}
This is to show that
\[
\frac{1}{2}\left(\frac{ T}{\sigma_2}+\frac{\tilde T}{\tilde \sigma_2}\right)\geq \frac{\bar T}{\bar \sigma_2}
\]
It completes the proof given the convexity of $H(r, Y)$ on $(r, Y)$ for $r\in \Gamma^+_2$, where
\begin{equation}\label{h}
H(r, Y):=\frac{T_1(r)(Y, Y)}{\sigma_2(r)}=\left(\frac{\p \log \sigma_2}{\p r_{ij}}\right) (Y, Y),
\end{equation}
The convexity of $H$ will be proved in the following. 
\end{proof}

\begin{thm}\label{convexity}The function $H(r, Y)$ in \eqref{h} is convex on $(r, Y)$ for $r\in \Gamma^+_2$.
\end{thm}

Theorem \ref{convexity} should have its own interest. We conjecture this holds for general $k$. Note that $k=1$ is straightforward, and when $k=n$ it is an old result of Marcus \cite{Marcus}.

\begin{conj}Let $n\geq 4$. Suppose $r$ is a $n\times n$ symmetric matrix such that $r\in \Gamma^+_k$, for $3\leq k\leq n-1$, then 
$H_k(r, Y)$ is a convex function on $r, Y$.
\end{conj}

First we need the following results, which give a simple proof of the well-known concavity of $\sqrt{\sigma_2}$. 
\begin{lemma}\label{qua1}For $r, \tilde r$ such that $\sigma_1=\sigma_1(r), \tilde \sigma_1=\sigma_1(\tilde r)$ both are positive, then we have the following identity
\begin{equation}\label{identity1}
\sigma_1 \tilde \sigma_1=\sigma_2\frac{\tilde \sigma_1}{\sigma_1}+\tilde \sigma_2\frac{\sigma_1}{\tilde \sigma_1}+\frac{1}{2}\left(|r|^2 \frac{\tilde \sigma_1}{\sigma_1}+|\tilde r|^2\frac{\sigma_1}{\tilde \sigma_1}\right)
\end{equation}
Moreover, if $r, \tilde r\in \Gamma^+_2$ then we have, given any unit vector $V_1\neq 0, |V|=1$,
\begin{equation}\label{identity2}
\sigma_1 \tilde \sigma_1=\left(\sigma_2+ \frac{1}{2}(|r|^2-r_{11}^2)\right)\frac{\tilde T(V_1, V_1)}{T(V_1, V_1)}+\left(\tilde \sigma_2+\frac{1}{2} (|\tilde r|^2-\tilde r_{11}^2)\right) \frac{T(V_1, V_1)}{\tilde T(V_1, V_1)}+r_{11}\tilde r_{11},
\end{equation}
where we use the notations $r_{11}=r(V_1, V_1)$.
\end{lemma}

\begin{proof}We need the following,
\begin{equation}\label{s2-1}
\sigma_2(r)=\frac{1}{2}\left(\sigma_1^2-|r|^2\right)
\end{equation}
We write
\[
\sigma_1\tilde \sigma_1=\frac{1}{2}\left(\sigma_1^2 \frac{\tilde \sigma_1}{\sigma_1}+\tilde \sigma_1^2\frac{\sigma_1}{\tilde \sigma_1}\right)
\]
Using \eqref{s2-1} this proves \eqref{identity1}. Now we prove \eqref{identity2}.
We choose an orthonormal basis  $\{V_1, V_2, \cdots, V_n\}$ which extends $V_1$. 
We write
\begin{equation*}
\begin{split}
\sigma_1 \tilde \sigma_1=&(\sigma_1- r_{11}+r_{11})(\tilde \sigma_1-\tilde r_{11}+\tilde r_{11})\\
=&(\sigma_1- r_{11})(\tilde \sigma_1-\tilde r_{11})+r_{11}(\tilde \sigma_1-\tilde r_{11})+\tilde r_{11}(\sigma_1- r_{11})+r_{11}\tilde r_{11}\\
=&\left[\frac{(\sigma_1-r_{11})^2}{2}+(\sigma_1-r_{11})r_{11}\right]\frac{\tilde \sigma_1-\tilde r_{11}}{\sigma_{1}-r_{11}}+\left[\frac{(\tilde \sigma_1-\tilde r_{11})^2}{2}+(\tilde \sigma_1-\tilde r_{11})\tilde r_{11}\right]\frac{\sigma_{1}-r_{11}}{\tilde \sigma_1-\tilde r_{11}}+r_{11}\tilde r_{11}
\end{split}
\end{equation*}
Then we compute
\[
\begin{split}
\sigma_2=&\frac{1}{2}\left(\sigma_1^2-|r|^2\right)\\
=& \frac{1}{2}\left[(\sigma_1-r_{11}+r_{11)})^2-|r|^2\right]\\
=&\frac{1}{2}(\sigma_1-r_{11})^2+(\sigma_{1}-r_{11})r_{11}+\frac{r_{11}^2-|r|^2}{2}
\end{split}
\]
The identity \eqref{identity2} follows by combining the above two computations directly. 
\end{proof}

Given $r, \tilde r$, a direct computation gives
\begin{equation}\label{sigma2}
4\bar\sigma_2=\sigma_2+\tilde \sigma_2+\sigma_1\tilde \sigma_1-(r, \tilde r)
\end{equation}
We denote $Q:=\sigma_1\tilde \sigma_1-(r, \tilde r)$, then we have
\[
Q=\sigma_2\frac{\tilde \sigma_1}{\sigma_1}+\tilde \sigma_2\frac{\sigma_1}{\tilde \sigma_1}+\frac{1}{2}\left(|r|^2 \frac{\tilde \sigma_1}{\sigma_1}+|\tilde r|^2\frac{\sigma_1}{\tilde \sigma_1}\right)-(r, \tilde r)
\]
In particular, this proves that, for $r\in \Gamma^+_2$
\[
Q\geq \sigma_2\frac{\tilde \sigma_1}{\sigma_1}+\tilde \sigma_2\frac{\sigma_1}{\tilde \sigma_1}\geq 2\sqrt{\sigma_2\tilde \sigma_2}. 
\]
This implies in particular the well-known  concavity of $\sqrt{\sigma_2}$, 
\[
\begin{split}
4\bar \sigma_2=&\sigma_2+\tilde \sigma_2+Q\\
\geq &\sigma_2+\tilde \sigma_2+2\sqrt{\sigma_2\tilde \sigma_2}\\
=&(\sqrt{\sigma_2}+\sqrt{\tilde \sigma_2})^2
\end{split}
\]
But we will need the full strength of \eqref{identity2}. We rewrite it as the following, for any given unit vector $V_1\in \R^n, |V_1|=1$,
\begin{equation}\label{Q1}
Q=\sigma_2 \frac{\tilde T_{11}}{T_{11}}+\tilde \sigma_{2}\frac{T_{11}}{\tilde T_{11}}+M_1,
\end{equation}
where we use the notation $T(V_1, V_1)=T_{11}$ and 
\[
M_1=\left(\frac{|r|^2}{2}\frac{\tilde T_{11}}{T_{11}}+\frac{|\tilde r|^2}{2}\frac{T_{11}}{\tilde T_{11}}-(r, \tilde r)\right)-\left(\frac{|r_{11}|^2}{2}\frac{\tilde T_{11}}{T_{11}}+\frac{|\tilde r_{11}|^2}{2}\frac{T_{11}}{\tilde T_{11}}-r_{11}\tilde r_{11}\right)
\]
Clearly $M_1\geq 0$ and it leads to
\begin{equation}\label{Q2}
Q\geq \sigma_2 \frac{\tilde T_{11}}{T_{11}}+\tilde \sigma_{2}\frac{T_{11}}{\tilde T_{11}}.
\end{equation}
Indeed, we need the following,
\begin{lemma}\label{Q100}
For any two unit vectors $V_1, W$, we have
\begin{equation}\label{key-202}
Q\geq \sigma_2 \frac{\tilde T(V_1, V_1)}{T(V_1, V_1)}+\tilde \sigma_2 \frac{T(V_1, W)^2}{T(V_1, V_1) \tilde T(W, W)}.
\end{equation}
Similarly we have
\begin{equation}\label{key-303}
Q\geq \tilde \sigma_2  \frac{T(V_1, V_1)}{\tilde T(V_1, V_1)}+ \sigma_2 \frac{\tilde T(V_1, W)^2}{\tilde T(V_1, V_1) T(W, W)}.
\end{equation}
\end{lemma}

\begin{proof}By symmetry we only prove \eqref{key-202}, while \eqref{key-303} follows by switching $T$ and $\tilde T$. 
If $W=V_1$ or $-V_1$, then \eqref{key-202} reduces to \eqref{Q2}. If $V_1$ and $W$ are linearly independent, we write 
\[
W=xV_1+yV_2
\]
for two orthogonal unit vectors $V_1$ and $V_2$. We choose a basis $\{V_1, V_2, \cdots, V_n\}$ as an extension of $\{V_1, V_2\}$ and we write $r=(r_{ij}), \tilde r=(\tilde r_{ij})$ in terms of this basis. By the homogeneity, we can assume (otherwise choose a scaling $r\rightarrow a r$ for some $a>0$)
\begin{equation}
\label{nort}
T_{11}=\tilde T_{11}.
\end{equation} With this normalization \eqref{nort}, we have by \eqref{Q1}
\[
Q=\sigma_2+\tilde \sigma_2+\frac{1}{2}\sum_{(ij)\neq (11)} \left(\tilde r_{ij}-r_{ij}\right)^2
\]
We compute
\[
T(V_1, W)=xT_{11}+yT_{12}, \tilde T(W, W)=x^2\tilde T_{11}+2xy\tilde T_{12}+y^2 \tilde T_{22})
\]
Hence we need to show that, using \eqref{Q1}, 
\begin{equation}\label{k0}
\tilde \sigma_2+\frac{1}{2}\sum_{(ij)\neq (11)} \left(\tilde r_{ij}-r_{ij}\right)^2\geq \frac{ \tilde \sigma_2(xT_{11}+yT_{12})^2}{T_{11} (x^2\tilde T_{11}+2xy\tilde T_{12}+y^2 \tilde T_{22})},
\end{equation}
We have also $T_{12}=-r_{12}, \tilde T_{12}=-\tilde r_{12}$. Next we compute
\begin{equation}\label{k1}
\begin{split}
\tilde T_{11}\tilde T_{22}=&(\tilde \sigma_1-\tilde r_{11})(\tilde \sigma_1-r_{22})\\
=&(\tilde r_{11}+\tilde r_{33}+\cdots+\tilde r_{nn})(\tilde r_{22}+\tilde r_{33}+\cdots \tilde r_{nn})\\
=&\tilde r_{11}\tilde r_{22}+\tilde r_{11}(\tilde r_{33}+\cdots \tilde r_{nn})+\tilde r_{22}(\tilde r_{33}+\cdots+\tilde r_{nn})+(\tilde r_{33}+\cdots \tilde r_{nn})^2\\
=&\sigma_2(\tilde r_{11}, \cdots, \tilde r_{nn})+(\tilde r_{33}^2+\cdots+\tilde r_{nn}^2)+\sigma_2(\tilde r_{33}, \cdots, \tilde r_{nn})
\end{split}
\end{equation}
Clearly we have
\begin{equation}\label{k2}
(\tilde r_{33}^2+\cdots+\tilde r_{nn}^2)+\sigma_2(\tilde r_{33}, \cdots, \tilde r_{nn})=\frac{1}{2}(\tilde r_{33}^2+\cdots+\tilde r_{nn}^2)+\frac{1}{2}\left(\tilde r_{33}+\cdots \tilde r_{nn}\right)^2
\end{equation}
We also have
\begin{equation}\label{k3}
\begin{split}
2\sigma_2(\tilde r_{11}, \cdots, \tilde r_{nn})=&\sigma_1((\tilde r_{11}, \cdots, \tilde r_{nn})- (r_{11}^2+\cdots \tilde r_{nn}^2)\\
=&\tilde \sigma_1^2- (r_{11}^2+\cdots \tilde r_{nn}^2)\\
=&2\tilde \sigma_2+|\tilde r|^2-(r_{11}^2+\cdots \tilde r_{nn}^2)\\
=&2\tilde \sigma_2+\sum_{i\neq j} \tilde r_{ij}^2\\
\geq & 2\tilde \sigma_2+2\tilde r_{12}^2
\end{split}
\end{equation}
Put \eqref{k1}, \eqref{k2} and \eqref{k3} together, we reach the conclusion that
\begin{equation}\label{k4}
\tilde T_{11}\tilde T_{22}\geq \tilde \sigma_2+\tilde r_{12}^2
\end{equation}
Clearly we have
\[
\frac{1}{2}\sum_{(ij)\neq (11)} \left(\tilde r_{ij}-r_{ij}\right)^2\geq (\tilde r_{12}-r_{12})^2.
\]
We claim
\begin{equation}\label{k5}
[\tilde \sigma_2+(\tilde r_{12}-r_{12})^2](x^2 T_{11}^2-2xy T_{11}\tilde r_{12}+y^2 \tilde T_{11}\tilde T_{22})\geq \tilde \sigma_2(xT_{11}-yr_{12})^2
\end{equation}
Given the claim \eqref{k5} and the normalization condition \eqref{nort}, this completes the proof of \eqref{k0}. While the claim \eqref{k5} is a direct consequence of \eqref{k4} and Cauchy-Schwartz inequality as follows. We compute, using \eqref{k4},
\[
x^2 T_{11}^2-2xy T_{11}\tilde r_{12}+y^2 \tilde T_{11}\tilde T_{22}\geq (xT_{11}- y\tilde r_{12})^2+\tilde \sigma_2 y^2,
\]
hence 
\[
\begin{split}
[\tilde \sigma_2+(\tilde r_{12}-r_{12})^2](x^2 T_{11}^2-2xy T_{11}\tilde r_{12}+y^2 \tilde T_{11}\tilde T_{22})\geq& [\tilde \sigma_2+(\tilde r_{12}-r_{12})^2][(xT_{11}-y\tilde r_{12})^2+y^2\tilde \sigma_2]\\
\geq &
\left(\sqrt{\tilde \sigma_2}(xT_{11}-y\tilde r_{12})+(\tilde r_{12}-r_{12}) \sqrt{\tilde\sigma_2} y\right)^2\\
=&\tilde \sigma_2(xT_{11}-yr_{12})^2\end{split}
\]
This completes the proof. 
\end{proof}

Now we are ready to prove the convexity of $H(r, Y)$. 
\begin{proof}This is to show
\[
\frac{\bar \sigma_2 T}{\sigma_2}+\frac{\bar \sigma_2\tilde T}{\tilde \sigma_2}\geq 2\bar T.
\]
First we assume that $r$ and $\tilde r$ commute and hence can be diagonalized simultaneously with eigenvalues $\l_1, \cdots, \l_n$ and $\tilde \l_1, \cdots, \tilde \l_n$.
We do not  order the eigenvalues at this point (since we cannot order the eigenvalues simultaneously). Writing $Y=(y_i), \tilde Y=(\tilde y_i)$, we  need to show that
\[
\frac{\bar \sigma_2 }{\sigma_2} \sum_i (\sigma_1-\l_i) y_i^2+\frac{\bar \sigma_2}{\tilde \sigma_2} \sum_i (\tilde \sigma_1-\tilde \l_i) \tilde y_i^2\geq \sum_i \left(\sigma_1-\l_i+\tilde \sigma_1-\tilde \l_i\right) \left(\frac{y_i+\tilde y_i}{2}\right)^2
\]
It is sufficient to show that, for each fixed $i$, we have
\[
\frac{\bar \sigma_2 }{\sigma_2}  (\sigma_1-\l_i) y_i^2+\frac{\bar \sigma_2}{\tilde \sigma_2} (\tilde \sigma_1-\tilde \l_i) \tilde y_i^2\geq  \left(\sigma_1-\l_i+\tilde \sigma_1-\tilde \l_i\right) \left(\frac{y_i+\tilde y_i}{2}\right)^2
\]
We take $i=1$ and write $y_i=y$ etc for simplicity.  We need to show,
\begin{equation}\label{dia1}
\begin{split}
\left(\frac{4\bar \sigma_2 }{\sigma_2}  (\sigma_1-\l_1)-(\sigma_1-\l_1)-(\tilde \sigma_1-\tilde \l_1)\right) y^2+&\left(\frac{4\bar \sigma_2 }{\tilde \sigma_2}  (\tilde \sigma_1-\tilde \l_1)-(\sigma_1-\l_1)-(\tilde \sigma_1-\tilde \l_1)\right) \tilde y^2\\
\geq & 2(\sigma_1-\l_1+\tilde \sigma_1-\tilde \l_1) y\tilde y.
\end{split}
\end{equation}
Denote for now $A=\sigma_1-\l_1+\tilde \sigma_1-\tilde \l_1$. We claim the following two inequalities,
\begin{equation}\label{simplecase}
\begin{split}
&\frac{4\bar \sigma_2 }{\sigma_2}  (\sigma_1-\l_1)-(\sigma_1-\l_1)-(\tilde \sigma_1-\tilde \l_1)>0\\
&\left(\frac{4\bar \sigma_2 }{\sigma_2}  (\sigma_1-\l_1)-A\right)\left(\frac{4\bar \sigma_2 }{\tilde \sigma_2}  (\tilde \sigma_1-\tilde \l_1)-A)\right)\geq A^2
\end{split}
\end{equation}
Given the claim this completes the proof of \eqref{dia1}. Now we establish \eqref{simplecase}. 
The first inequality in \eqref{simplecase} is a direct consequence of \eqref{sigma2} and \eqref{Q1}.  For the second inequality in \eqref{simplecase}, by a direct computation, we need to show that
\[
4\bar \sigma_2 (\sigma_1-\l_1)(\tilde \sigma_1-\tilde \l_1)\geq A\left(\tilde \sigma_2  (\sigma_1-\l_1)+\sigma_2(\tilde \sigma_1-\tilde \l_1)\right)
\]
That is to show 
\[
(\sigma_2+\tilde \sigma_2+\sigma_1\tilde \sigma_1-\sum\l_i\tilde \l_i)(\sigma_1-\l_1)(\tilde \sigma_1-\tilde \l_1)\geq A\left(\tilde \sigma_2  (\sigma_1-\l_1)+\sigma_2(\tilde \sigma_1-\tilde \l_1)\right)
\]
In other words,
\[
(\sigma_1\tilde \sigma_1-\sum\l_i\tilde \l_i)(\sigma_1-\l_1)(\tilde \sigma_1-\tilde \l_1)\geq \tilde \sigma_2  (\sigma_1-\l_1)^2+\sigma_2(\tilde \sigma_1-\tilde \l_1)^2.\]
This is a direct consequence of \eqref{Q1} and \eqref{Q2}, with $T_{11}=\sigma_1-\l_1$.  This completes the proof when $r$ and $\tilde r$ can be diagonalized simultaneously (when $r\tilde r=\tilde r r$).\\

Next we consider the general case.  We compute, noting that $T_1(r)$ is a linear operator on $r$,
\begin{equation}
\begin{split}
2\bar T=&T(r)(\bar Y, \bar Y)+T(\tilde r)(\bar Y, \bar Y)\\
=&\frac{1}{4} \left(T(Y, Y)+T(\tilde Y, \tilde Y)\right)+\frac{1}{4} \left(\tilde T(Y, Y)+\tilde T(\tilde Y, \tilde Y)\right)+\frac{1}{2}\left(T(Y, \tilde Y)+\tilde T(Y, \tilde Y)\right)\\
=&\frac{1}{4} \left(T(Y, Y)+\tilde T(Y, Y)\right)+\frac{1}{4} \left(T(\tilde Y, \tilde Y)+\tilde T(\tilde Y, \tilde Y)\right)+\frac{1}{2}\left(T(Y, \tilde Y)+\tilde T(Y, \tilde Y)\right)
\end{split}
\end{equation}
Hence we need to show
\begin{equation}\label{hardcase}
\left(\frac{4\bar \sigma_2 T}{\sigma_2}-T-\tilde T\right)(Y, Y)+\left(\frac{4\bar \sigma_2 \tilde T}{\tilde \sigma_2}-T-\tilde T\right)(\tilde Y, \tilde Y)\geq 2\left(T+\tilde T\right)(Y, \tilde Y).
\end{equation}
Note that the following matrices are positive definite, as a direct consequence of \eqref{Q1} and \eqref{Q2},
\begin{equation}\label{hard-positive}
\frac{4\bar \sigma_2 T}{\sigma_2}-T-\tilde T>0,\; \frac{4\bar \sigma_2 \tilde T}{\tilde \sigma_2}-T-\tilde T>0
\end{equation}
We assume $Y, \tilde Y\neq 0$ (otherwise we are done by the positivity \eqref{hard-positive}). 
We want to prove the following,  
\begin{equation}\label{hard-202}
\left(\frac{4\bar \sigma_2}{\sigma_2}T-T-\tilde T\right)(Y, Y)\left(\frac{ 4\bar\sigma_2}{\tilde \sigma_2}\tilde T- T-\tilde T\right)(\tilde Y, \tilde Y)\geq (T(Y, \tilde Y)+\tilde T(Y, \tilde Y))^2. 
\end{equation}
Clearly \eqref{hardcase} is a direct consequence of \eqref{hard-202}. 
By homogeneity, we can require $|Y|=|\tilde Y|=1$. 
Since $4\bar \sigma_2=\sigma_2+\tilde \sigma_2+Q$, we need to show
\[
\left(\frac{\tilde \sigma_2+Q}{\sigma_2}T-\tilde T\right)(Y, Y)\left(\frac{ \sigma_2+Q}{\tilde \sigma_2}\tilde T- T\right)(\tilde Y, \tilde Y)\geq (T(Y, \tilde Y)+\tilde T(Y, \tilde Y))^2. 
\]
First we apply Lemma \ref{Q100}, using \eqref{key-202} with $V_1=Y, W=\tilde Y$, hence we get
\[
Q\geq \sigma_2 \frac{\tilde T(Y, Y)}{T(Y, Y)}+\tilde \sigma_2 \frac{T(Y, \tilde Y)^2}{T(Y, Y)\tilde T(\tilde Y, \tilde Y)}. 
\]
It follows that
\begin{equation}\label{Q200}
\left(\frac{\tilde \sigma_2+Q}{\sigma_2}T-\tilde T\right)(Y, Y)\geq \frac{\tilde \sigma_2}{\sigma_2} \left(T(Y, Y)+\frac{T(Y, \tilde Y)^2}{\tilde T(\tilde Y, \tilde Y)}\right). 
\end{equation}
Then we apply Lemma \ref{Q100}, using \eqref{key-303} with $V_1=\tilde Y, W=Y$ hence we get
\[
Q\geq \tilde \sigma_2 \frac{T(\tilde Y, \tilde Y)}{\tilde T(\tilde Y, \tilde Y)}+\sigma_2 \frac{\tilde T(Y, \tilde Y)^2}{\tilde T(\tilde Y, \tilde Y)T(Y, Y)}
\]
It follows that
\begin{equation}\label{Q300}
\left(\frac{ \sigma_2+Q}{\tilde \sigma_2}\tilde T- T\right)(\tilde Y, \tilde Y)\geq \frac{\sigma_2}{\tilde \sigma_2}\left(\tilde T(\tilde Y, \tilde Y)+\frac{\tilde T(Y, \tilde Y)^2}{T(Y, Y)}\right)
\end{equation}
Put \eqref{Q200} and \eqref{Q300} together, we have
\[
\begin{split}
\left(\frac{\tilde \sigma_2+Q}{\sigma_2}T-\tilde T\right)(Y, Y)&\left(\frac{ \sigma_2+Q}{\tilde \sigma_2}\tilde T- T\right)\geq  \left(T(Y, Y)+\frac{T(Y, \tilde Y)^2}{\tilde T(\tilde Y, \tilde Y)}\right)\left(\tilde T(\tilde Y, \tilde Y)+\frac{\tilde T(Y, \tilde Y)^2}{T(Y, Y)}\right)\\
=&T(Y, Y)\tilde T(\tilde Y, \tilde Y)+T(Y, \tilde Y)^2+\tilde T(Y, \tilde Y)^2+\frac{\tilde T(Y, \tilde Y)^2T(Y, \tilde Y)^2}{T(Y, Y)\tilde T(\tilde Y\tilde Y)}\\
\geq &T(Y, \tilde Y)^2+\tilde T(Y, \tilde Y)^2+2T(Y, \tilde Y)\tilde T(Y, \tilde Y)\\
=& \left(T(Y, \tilde Y)+\tilde T(Y, \tilde Y)\right)^2
\end{split}
\]
This proves \eqref{hard-202} hence it completes the proof. 
\end{proof}

\begin{rmk}It would be attempting to use Garding's theory of hyperbolic polynomials to demonstrate the concavity of $F^{\frac{1}{3}}$, which is slightly stronger than the concavity of $\log F$. This is to show that the following cubic equation has only real roots, for any symmetric matrix $R$. The cubic equation (in $t$)  reads
\[
F(R+tJ)=0,
\]
where the matrix $J$ can be taken as $I_{n+1}$ (or the matrix $I_3$, viewed as a symmetric $(n+1)\times (n+1)$ matrix by an obvious embedding).
Even though it is a standard process to check when a cubic polynomial has real roots and we believe this is correct for our setting. But the computation is quite involved and we are not able to carry out this approach directly. 
\end{rmk}

\subsection{The metric structure and the uniqueness of $\sigma_2$-Yamabe problem when $n=4$}
Given the $C^{1, 1}$ regularity,  the formal metric picture of Gursky-Streets \cite{GS2} can be made strict; moreover the proof of the uniqueness of $\sigma_2$-Yamabe problem can be made much more straightforward. First we summarize some direct consequence of the existence of $C^{1, 1}$ geodesic for Gursky-Streets metric. 

We fix some notations. Consider the approximating geodesic equation, given two fixed boundary datum $u_0, u_1$,
\[
u_{tt}\sigma_2(A_u)-\langle T_1(A_u), \nabla u_t\otimes \nabla u_t\rangle=sf.
\]
We have obtained uniform $C^{1, 1}$ estimates for any smooth $f>0$. We take $f\equiv 1$ in particular to get an approximating geodesic $u^s$ and denote $u$ to be its limit. We refer $u$ as the geodesic connecting $u_0, u_1$. 

 \begin{thm}Let $(M, g)$ be a compact Riemannian manifold of dimension four such that $\cC^+\neq \emptyset$. Then $\cC^+$ is a metric space with Gursky-Streets metric. Given $u_0, u_1\in \cC^+$, the geodesic realizes the distance between $u_0, u_1$. In particular $\cC^+$ has nonpositive curvature in the sense of Alexanderov. 
 \end{thm}

The argument is rather standard (but a bit long and tedious), given the formal geometric picture verified by Gursky-Streets [Section 3]\cite{GS2} with smooth geodesics. The main point is to use the approximating geodesic $u^s$ instead of the limit geodesic $u$ since $u^s$ is smooth and is admissible. All the identities hold modulo quantities in the order of $O(s)$ given the uniform $C^{1, 1}$ regularity; the results then follow by taking $s\rightarrow 0$.   (See [Section 5]\cite{Chen} and [Section 5]\cite{Chen-He}  for example). 
We skip the details since we do not really need these results. 
We will only verify the geodesic convexity of the functional $\cF$ of Chang-Yang and give an alternative proof of uniqueness of $\sigma_2$-Yamabe problem.
We will need the following curvature weighted Poincare-inequalities, due to B. Andrews \cite{Andrews}.

\begin{lemma}[Andrews]\label{andrews0}Let $(M^n, g)$ be a compact Riemannian manifold with positive Ricci curvature. Given a Lipschitz function $\phi$ with $\int_M \phi dv=0$, then
\[
\frac{n}{n-1}\int_M \phi^2 dv\leq \int_M (Ric^{-1}) (\nabla \phi, \nabla \phi)dv,
\]
with the equality if and only if $\phi\equiv 0$ or  $(M^n, g)$ is isometric to the round sphere. 
\end{lemma}
Gursky-Streets obtained a weaker form of this inequality for $n=4$, 
\begin{lemma}[Gursky-Streets \cite{GS2}] \label{andrews}Let $(M^4, g)$ be a closed Riemannian manifold such that $A_g\in \Gamma^+_2$. Given a Lipschitz function $\phi$, then
\[
\int_M \frac{1}{\sigma_2(A_g)} T_1(A_g)(\nabla \phi, \nabla \phi) dv\geq 4\int_M \phi^2 dv-\frac{4}{\int_M dv}\left(\int_M \phi dv\right)^2.
\]
The equality holds if and only if $\phi$ is a constant or $(M^4, g)$ is isometric to the round sphere. 
\end{lemma}
We can now verify the convexity of the functional $\cF$ along the $C^{1, 1}$ geodesic. Indeed $\cF$ is convex along the smooth approximating geodesic $u^s$ for any $s\in (0, 1]$.

\begin{thm}\label{con707}Given $u_0, u_1\in \cC^+$,  let $u^s$ be the approximating geodesic satisfying 
\[
u_{tt}\sigma_2(A_u)-\langle T_1(A_u), \nabla u_t\otimes \nabla u_t\rangle=s.
\]
Then $\cF$ is convex along the $C^{1, 1}$ geodesic $u$.  
In particular $\cF$ achieves its minimum energy at any smooth critical point. \end{thm}

\begin{proof}Let $u^s$ be the unique smooth solution of the equation,  
\begin{equation}\label{last3}u_{tt}\sigma_2(A_u)-\langle T_1(A_u), \nabla u_t\otimes \nabla u_t\rangle=s.
\end{equation}
Denote $u$ to be ``the geodesic", which is the limit of $u^s$ when $s\rightarrow 0$. Consider the functional $\cF(u)$ and $\cF(u^s)$ for $t\in [0, 1]$. 
By the uniform estimate, we know that $u^s$ converges to $u$ in $C^{1, \alpha}([0, 1]\times M)$ for any $\alpha\in [0, 1)$. Moreover, we compute
\[
\begin{split}
\int_M \Delta u |\nabla u|^2dV-\int_M \Delta u^s |\nabla u^s|^2 dV=&\int_M \Delta u (|\nabla u|^2-|\nabla u^s|^2)dV+\int_M |\nabla u^s|^2 \Delta (u-u^s) dV\\
=&\int_M \Delta u (|\nabla u|^2-|\nabla u^s|^2)dV+\int_M \nabla (|\nabla u^s|^2) \nabla(u-u^s) dV
\end{split}
\]
It follows that, $\int_M \Delta u^s |\nabla u^s|^2dV$ converges to $\int_M \Delta u |\nabla u|^2dV$ (uniformly with respect to $t$) when $s\rightarrow 0$.
Using the formula \eqref{F-functional}, it implies that $\cF(u^s)$ converges to $\cF(u)$ uniformly w.r.t $t$. In particular $\cF(u)$ is continuous w.r.t $t\in [0, 1]$. 
A similar argument shows that 
$\int_M \Delta u |\nabla u|^2 dV$
is Lipschitz in $t$ and hence $\cF(u)$ is Lipschitz in $t$. 
Denote the conformal invariant total $\sigma_2$ curvature as
\[
\sigma=\int_M \sigma_2(g_u^{-1}A_u)dV_u\;\text{and}\; \bar \sigma=\sigma V_u^{-1}
\]
where $V_u$ is the total volume of $g_u$. 
Along the path $u^s$, using the variational structure of $\cF$ \cite{BV} (see the computation as in \cite{GS2}), we have
\[
\frac{d\cF(u^s)}{dt}=\int_M u^s_t(-\sigma_2(g^{-1}_{u^s}A_{u^s})+\bar \sigma ) dV_{u^s}
\]
To compute the second derivative we need to be careful about the conformal factor. 
We compute the second derivative (using \eqref{p101}, Lemma \ref{free} and the equation \eqref{last3}), 
\begin{equation}\label{last1}
\begin{split}
\frac{d^2\cF(u^s)}{dt^2}=&\int_M \left(-u^s_{tt}\sigma_2({g_u^s}^{-1}A_{u^s})-u_t^s\langle T_1({g_{u^s}}^{-1}A_{u^s}), \nabla^2 u^s_t\rangle_{g_{u^s}}\right)dV_{u^s}\\
&+\bar \sigma \int_M\left[ u^s_{tt}-4\left(u^s_t-\underline {u^s_t}\right)^2 \right]dV_{u^s}\\
=&\int_M \left(-u_{tt}\sigma_2(A_{u^s})+\langle T_1(A_{u^s}), \nabla u^s_t\otimes \nabla u^s_t\rangle \right) dV+\bar \sigma \int_M\left[ u^s_{tt}-4\left(u^s_t-\underline {u^s_t}\right)^2 \right]dV_{u^s}\\
=&-s\int_M dV+\bar \sigma \int_M\left[ u^s_{tt}-4\left(u^s_t-\underline {u^s_t}\right)^2 \right]dV_{u^s},
\end{split}
\end{equation}
where we use the notation of average,  
\[
\underline {u^s_t}={V_{u^s}}^{-1}\int_M u^s_t dV_{u^s}.
\]
We compute, using the equation \eqref{last3}, 
\[
\int_M u^s_{tt} dV_{u^s}=\int_M \frac{1}{\sigma_2(g^{-1}_{u^s}A_{u^s})} \langle T_1(g_{u^s}^{-1}A_{u^s}), \nabla u^s_t\otimes \nabla u^s_t\rangle_{g_{u^s}} dV_{u^s}+s\int_M \frac{1}{\sigma_2(g^{-1}_{u^s}A_{u^s})}dV
\]
Hence it follows that
\begin{equation}\label{last2}
\begin{split}
\frac{d^2\cF(u^s)}{dt^2}=& -s\int_M dV+s\bar \sigma \int_M \frac{1}{\sigma_2(g^{-1}_{u^s}A_{u^s})}dV\\&+ \bar \sigma \int_M \left[\frac{1}{\sigma_2(g^{-1}_{u^s}A_{u^s})} \langle T_1(g_{u^s}^{-1}A_{u^s}), \nabla u^s_t\otimes \nabla u^s_t\rangle_{g_{u^s}} -4\left(u^s_t-\underline {u^s_t}\right)^2 \right]dV_{u^s}.
\end{split}
\end{equation}
By  Lemma \ref{andrews} we know that
\[
\frac{d^2\cF(u^s)}{dt^2}> -s \int_M dV.
\]
This shows the convexity of $\cF(u^s)+s t^2 \int_M dV$. Taking $s\rightarrow 0$, this implies the convexity of $\cF$ along the geodesic $u$. 
The second part of the statement follows directly. Note that the second part of the statement was verified by Gursky-Streets [Lemma 6.1]\cite{GS2}. 
\end{proof}

Now we suppose $u_0, u_1\in \cC^+$ are two smooth critical points of $\cF$. Then we have the following, 
\begin{cor}Let $u$ be the $C^{1, 1}$ geodesic connecting $u_0, u_1$. Then either $(M^4, g_{u^i})$ is isometric to the round sphere, or $u_1=u_0+c$ for some constant $c$.\end{cor}
\begin{proof}Since $\cF$ achieves its minimum at $u_0$ and $u_1$, by the convexity of $\cF$ we know that $\cF$ remains constant along the geodesic $u$. In other words, $u(t)$ minimizes $\cF$ for any $t\in [0, 1]$. 
We claim $u(t): M\rightarrow \R$ is smooth and is in $\cC^+$ for each $t$ and solves the equation $\sigma_2(A_u)=\text{const}.$ 
For simplicity we drop the dependence on $t$ since the argument is the same. Due to the only $C^{1, 1}$ regularity of $u$, the essential point is to prove that $\cF$ achieves a minimum at $u$ in a certain class $\bar\Gamma^+_2$ in the following sense.
 Suppose $u$ has $C^{1, 1}$ bound and  a sequence of $u^s\in \cC^{+}$ with uniform $C^{1, 1}$ bound converges to $u$. Then for any smooth $v\in \cC^+$, and $r\in (0, 1)$ sufficiently small, we have the following variational characteristic description of $u$ with respect to $\cF$-functional, \[\cF(u)\leq \cF(u+rv).\]
 Now we need to compute the first variation of $\cF$ at $u$. We need the following, at $r=0$,
 \begin{equation}\label{variation1}
 \frac{\p \cF(u+rv)}{\p r}=-\int_M v (\sigma_2(g_u^{-1}A_u)-\bar \sigma) dV_u
 \end{equation}
 If $u$ is smooth, then \eqref{variation1} follows directly \cite{BV}. A main point is that \eqref{variation1} holds using the fact  $T_1$ is divergence free (when $n=4$). When $u\in C^{1, 1}$, then $T_1(g_u^{-1}A_u)$ is divergence free in the following sense: for any smooth vector $X=(X^i)$, we have
 \begin{equation}\label{div1}
 \int_M \sum_j T_1(g_u^{-1}A_u) ^{ij} \nabla_j X^i dV_u=0
 \end{equation}
 We can choose a sequence of smooth function $u_n$ such that $u_n$ converges to $u$ in $W^{2, p}$ and $u_n$ has uniform $C^{1, 1}$ bound. A direct approximation argument gives \eqref{div1}. Given \eqref{div1}, \eqref{variation1} follows directly as in \cite{BV}; the point is that the following one-form $\alpha$ is still closed for $u\in C^{1, 1}$ and it gives the first variation of $\cF$, by the computation as in \cite{BV} together with \eqref{div1}, where
 \[
 \alpha (v)=-\int_M v (\sigma_2(g_u^{-1}A_u)-\bar \sigma) dV_u.
 \]
  Since $\cF(u)\leq \cF(u+rv)$, we have at $r=0$, for any $v$,
 \begin{equation}
  \frac{\p \cF(u+rv)}{\p r}=-\int_M v (\sigma_2(g_u^{-1}A_u)-\bar \sigma) dV_u\geq 0
 \end{equation}
 Since we can add any constant to $v$, this implies that $\sigma_2(g_u^{-1}A_u)-\bar \sigma=0$, where $\bar \sigma=\sigma V_u^{-1}$. It follows that $\sigma_2(g_u^{-1}A_u)>0$. Hence $u\in C^{1, 1}$ is a strong solution of the uniform elliptic equation
 \begin{equation}\label{ulast}
 \sigma_2(A_u)=\sigma V_u^{-1} e^{-4u}
 \end{equation}
 and  the standard elliptic estimate then gives the smoothness of $u$ (in space direction). Hence $u(t): M\rightarrow \R$ is smooth for each $t$ and it solves the equation \eqref{ulast}.
 Taking derivative with respect to $t$, the elliptic regularity then implies that $u_t$ is smooth in space direction. Note that we do not assert at the moment that $u$ is smooth in space time, even though we know this holds \emph{a posteriori}.
 Nevertheless we can directly compute, similar as in \eqref{last2},
 \[
 \frac{d^2 \cF(u(t))}{dt^2}=\bar \sigma \int_M \left[\frac{1}{\sigma_2(g^{-1}_{u}A_{u})} \langle T_1(g_{u}^{-1}A_{u}), \nabla u^s_t\otimes \nabla u_t\rangle_{g_{u}} -4\left(u_t-\underline {u_t}\right)^2 \right]dV_{u}=0.
 \]
 This implies that $u_t=\text{const}$ or $(M^4, g_u)$ is isometric to the round sphere $S^4$, by Lemma \ref{andrews}. 
\end{proof}

 This gives a direct proof of the uniqueness of $\sigma_2$-Yamabe problem.
 \begin{cor}Let $(M^4, g)$ be a compact four manifold with $\cC^+\neq \emptyset$. 
\begin{enumerate}\item There exists a unique solution to the $\sigma_2$-Yamabe problem in $[g]$ if $(M^4, g)$ is not conformally equivalent to the round $S^4$.

\item In $[g_{S^4}]$, all solutions to the $\sigma_2$-problem are round metrics.
\end{enumerate}
\end{cor}

\end{document}